\documentclass[a4paper,twoside]{article}  

\usepackage{amssymb,amsmath,amsthm,dsfont,amsfonts,color,latexsym}

\usepackage{hyperref}

\usepackage{mathrsfs} 

\definecolor{blue}{rgb}{0.00,0.00,1.00}
\definecolor{red}{rgb}{1.00,0.00,0.00}

\allowdisplaybreaks[4]
\renewcommand{\baselinestretch}{1.2}
\def\bq{\begin{equation}}
\def\eq{\end{equation}}
\def\ba{\begin{array}{ccc}}
\def\bal{\begin{array}{lll}}
\def\ea{\end{array}}
\def\bsp{\begin{split}}
\def\esp{\end{split}}

 \def\lt#1{\left#1}\def\rt#1{\right#1}
\def\({\left(}\def\){\right)}
\def\[{\left[}\def\]{\right]}

    \def \R   {\mathbb{R}}

    \def\i    {\mathrm{i}}

    \def\S    {\mathbb{S}}
    \def\eps  {\epsilon}
    \def\intr {\int_{\R^3}}

    \def \N    {\mathbb{N}}

    \def \pt   {\partial}
    
    \def \Dt   {\frac{\rm d}{{\rm d}t}}

    \def \div  {{\rm div}}

    \def\Tdx   {\nabla_x}

       \def\bq{\begin{equation}}
       \def\eq{\end{equation}}
       \def\be{\begin{equation}}
       \def\ee{\end{equation}}
       \def\bma#1\ema{{\allowdisplaybreaks\begin{align}#1\end{align}}}
       \def\bmas#1\emas{{\allowdisplaybreaks\begin{align*}#1\end{align*}}}
       \def\bln#1\eln{{\allowdisplaybreaks\begin{aligned}#1\end{aligned}}}
       \def\nnm{\notag}
       \def\bgr#1\egr{\allowdisplaybreaks\begin{gather}#1\end{gather}}
       \def\bgrs#1\egrs{\allowdisplaybreaks\begin{gather*}#1\end{gather*}}

       \theoremstyle{plain}
       \newtheorem{lem}{\bf Lemma}[section]
       \newtheorem{thm}[lem]{\textbf{Theorem}}

       \newtheorem{remark}[lem]{\bf Remark}

\begin{document}


\title{Diffusion Limit with Optimal Convergence Rate of Classical Solutions to the modified Vlasov-Poisson-Boltzmann System}

\author{ Yanchao Li$^{1,2}$,\, Mingying Zhong$^{1}$ \\[2mm]
 \emph
    {\small\it $^1$School of  Mathematics and Information Sciences,
    Guangxi University,  China.}\\
    {\small\it  \& Center for Applied Mathematical of Guangxi (Guangxi University), Guangxi University,  China.}\\
    {\small\it E-mail:\ zhongmingying@gxu.edu.cn}\\
     {\small\it $^2$Post-doctoral Research Station of the First-level Discipline in Mathematics, Guangxi University,  China.}\\
    {\small\it E-mail:\ yanchaoli@gxu.edu.cn}}
\date{ }

\pagestyle{myheadings}
\markboth{modified Vlasov-Poisson-Boltzmann System}%
{Y.-C. Li, M.-Y. Zhong}

 \maketitle

 \thispagestyle{empty}

\begin{abstract}\noindent{
In the present paper, we study the diffusion limit of the classical solution to the modified Vlasov-Poisson-Boltzmann (mVPB) System with initial data near a global Maxwellian. Based on the spectral analysis, we prove the convergence and establish the convergence rate of the global strong solution to the mVPB system towards the solution to an incompressible Navier-Stokes-Poisson-Fourier system with the precise estimation on the initial layer.
}

\medskip
 {\bf Key words:
 }modified Vlasov-Poisson-Boltzmann system, spectral analysis, diffusion limit, convergence rate, initial layer.

\medskip
 {\bf 2010 Mathematics Subject Classification}. 76P05, 82C40, 82D05.
\end{abstract}

\tableofcontents

\section{Introduction}
\label{sect1}
\setcounter{equation}{0}
We consider the rescaled modified Vlasov-Poisson-Boltzmann (mVPB) system  \cite{Li-4,Markowich-1}:
\bma
&\partial_t F_{\epsilon}+\frac{1}{\epsilon}v\cdot\nabla_xF_{\epsilon}+\frac{1}{\epsilon}\nabla_x\Phi_{\epsilon}\cdot\nabla_v F_{\epsilon}=\frac{1}{\epsilon^2}Q(F_{\epsilon},F_{\epsilon}),\label{m1}\\
&\Delta_x\Phi_{\epsilon}=\int_{\mathbb{R}^3}F_{\epsilon}dv-e^{-\Phi_{\epsilon}},\label{m2}
\ema
where $\epsilon>0$ is a small parameter related to the mean free path, $F_{\epsilon}=F_{\epsilon}(t,x,v)$ is the distribution function with $(t,x,v)\in\mathbb{R}^{+}\times\mathbb{R}^3\times\mathbb{R}^3$, and $\Phi_{\epsilon}(t,x)$ denotes the electric potential. Since we study the diffusive limit when $\epsilon $ tends to zero, we assume $\eps\in(0,1)$ in the following analysis. The collision between particles is given by the standard Boltzmann collision operator $Q(f,g)$ as below
\bq
Q(f,g)=\frac{1}{2}\int_{\mathbb{R}^3}\int_{\mathbb{S}^2}B(|v-v_{\ast}|,\omega)\big(f'_{\ast}g'+f'g'_{\ast}-f_{\ast}g-fg_{\ast}\big)dv_{\ast}d\omega,
\eq
where
\bmas
&f'_{\ast}=f(t,x,v'_{\ast}),\quad f'=f(t,x,v'),\quad f_{\ast}=f(t,x,v_{\ast}),\quad f=f(t,x,v),\\
&v'=v-[(v-v_{\ast})\cdot\omega]\omega,\quad v_{\ast}'=v_{\ast}-[(v-v_{\ast})\cdot\omega]\omega, \quad\omega\in\mathbb{S}^2.
\emas
The collision kernel $B(|v-v_{\ast}|,\omega)$ is a non-negative function of $|v-v_{\ast}|$ and $|(v-v_{\ast})\cdot\omega|$:
\bq
B(|v-v_{\ast}|,\omega)=B(|v-v_{\ast}|,\cos\theta),\quad \cos\theta=\frac{|(v-v_{\ast})\cdot\omega|}{|v-v_{\ast}|}, \quad \theta\in\[0,\frac{\pi}{2}\].
\eq
In the following, we consider both the hard sphere model and hard potential model with angular cutoff. Precisely, for the hard sphere model,
\bq
B(|v-v_{\ast}|,\omega)=|(v-v_{\ast})\cdot\omega|=|v-v_{\ast}|\cos\theta,\quad  \theta\in\[0,\frac{\pi}{2}\],
\eq
and for the models of the hard potentials with Grad angular cutoff assumption,
\bq
B(|v-v_{\ast}|,\omega)=b(\cos\theta)|v-v_{\ast}|^{\gamma},\quad0\leq\gamma<1,\quad \theta\in\[0,\frac{\pi}{2}\],
\eq
where we assume for simplicity
$$
0\leq b(\cos\theta)\leq C|\cos\theta|.
$$

There has been significant progress made on the well-posedness and long time behavior of solutions to the Vlasov-Poisson-Boltzmann (VPB) system for fixed $\epsilon$. In particular, the global existence of renormalized weak solution for general large initial data was shown in \cite{Mischler-1}. The global existence of unique strong solution with the initial data near the normalized global Maxwellian was obtained in a spatially periodic domain in \cite{Guo-2} and in the whole space in \cite{Duan-2,Yang-1,Yang-3} for hard sphere model, and further in \cite{Duan-3,Duan-4} for hard potential or soft potential. As for the long-time behavior of the VPB system, we refer to the works \cite{Duan-1,Li-2,Li-4,Wang-2,Yang-2}. The existence of a classical solution near a vacuum was investigated in \cite{Guo-3}. In addition, the spectrum structure and the optimal decay rate of the classical solution were investigated in \cite{Li-2,Li-4}.

 The fluid dynamical limit of the solution to the VPB system near Maxwellian was also studied in \cite{Guo-1,Gong-1,Jiang-1,Li-1,Tong-1,Wang-1,Wu-1}. In \cite{Guo-1}, the authors proved the convergence of the solutions to the unipolar VPB system towards a solution to the compressible Euler-Poisson system in the whole space. In \cite{Wang-1}, the author established a global convergence result of the solution to the bipolar VPB system towards a solution to the incompressible Vlasov-Navier-Stokes-Fourier system. Moreover, the authors  justified the incompressible Navier-Stokes-Poisson-Fourier limit of  the VPB system in \cite{Gong-1,Jiang-1,Li-1,Tong-1}, and the incompressible Euler-Poisson limit in  \cite{Wu-1}.

On the other hand, the diffusion limit to the Boltzmann equation is a classical problem with pioneer work by Bardos-Golse-Levermore in \cite{Bardos-1}, and significant progress on the limit of renormalized solutions to Leray solution to Navier-Stokes system in \cite{Golse-1}.
One effective approach to study the fluid dynamic in the perturbative framework is based on the spectral analysis. For example,  Ellis-Pinsky \cite{Ellis-1} first studied the linear compressible Euler limit of the linear Boltzmann equation and showed the convergence rate outside the initial layer. The initial layer in the fluid limit arises from the incompatibility of the initial data, in particular due to the high oscillation of the eigen-modes in the system. Bardos-Ukai \cite{Bardos-3} firstly studied  the incompressible Navier-Stokes limit with the estimation on the initial layer for the linear  Boltzmann equation. The analysis in \cite{Bardos-3} uses an estimate on the semigroup with highly oscillating eigen-modes in \cite{Ukai-1} which is about the incompressible limit of  the compressible Euler equation. For the VPB system, Li-Yang-Zhong \cite{Li-1} first established the convergence rate of global strong solution towards the solution to an incompressible Navier-Stokes-Poisson-Fourier system, and provided a precise structure of the initial layer.
In contrast to the works on Boltzmann equation \cite{Bardos-1,Bardos-2,Bardos-3,Guo-4,Masi-1,Nishida-1} and the VPB system \cite{Guo-1,Gong-1,Jiang-1,Li-1,Wang-1}, the convergence and the optimal convergence rate of  the classical solution to the mVPB system  towards its fluid dynamical limit  and the estimation of the initial layer  have not been given.

In this paper, we study the diffusion limit of the strong solution to the rescaled mVPB system \eqref{m1}-\eqref{m2} with initial data near the normalized Maxwellian $M$, where
$$
M=M(v)=\frac{1}{(2\pi)^{\frac{3}{2}}}e^{-\frac{|v|^2}{2}},\quad v\in\mathbb{R}^3.
$$
Hence, we denote the perturbation of $F_{\epsilon}$ and $\Phi_{\epsilon}$ as
$$
F_{\epsilon}=M+\epsilon\sqrt{M}f_{\epsilon},\quad \Phi_{\epsilon}=\epsilon\phi_{\epsilon}.
$$
Thus, the rescaled mVPB system \eqref{m1}-\eqref{m2} for $f_{\epsilon}$ and $\phi_{\epsilon}$ is
\bma
&\partial_tf_{\epsilon}+\frac{1}{\epsilon}v\cdot\nabla_xf_{\epsilon}+\frac{1}{\epsilon}v\sqrt{M}\cdot\nabla_x(I-\Delta_x)^{-1}\int_{\R^3}f_{\epsilon}\sqrt{M}dv-\frac{1}{\epsilon^2}Lf_{\epsilon}\nnm\\
&\quad=G_1(f_{\epsilon})+\frac{1}{\epsilon}G_2(f_{\epsilon})+\frac{1}{\epsilon^2}v\sqrt{M}\cdot\nabla_x(I-\Delta_x)^{-1}G_3(f_{\epsilon}),\label{Pm1}\\
&(I-\Delta_x)\phi_{\epsilon}=-\int_{\R^3}f_{\epsilon}\sqrt{M}dv+\frac{1}{\epsilon}G_3(f_{\epsilon}),\label{Pm2}
\ema
with the initial condition
\bq
f_{\epsilon}(x,v,0)=f_{0}(x,v),\label{Pm3}
\eq
by assuming that the initial data $f_0$ is independent of $\epsilon$. As usual, the linear operator $L$ and the nonlinear operators $G_j$ $(j=1,2,3)$ are given by
\bq\label{3dmvpbnlt}
\left\{\bln
&Lf_{\epsilon}=\frac{1}{\sqrt{M}}[Q(M,\sqrt{M}f_{\epsilon})+Q(\sqrt{M}f_{\epsilon},M)],\\
&G_1(f_{\epsilon})=\frac{1}{2}(v\cdot\nabla_x\phi_{\epsilon})f_{\epsilon}-\nabla_x\phi_{\epsilon}\cdot\nabla_vf_{\epsilon},\\
&G_2(f_{\epsilon})=\Gamma(f_{\epsilon},f_{\epsilon})=\frac{1}{\sqrt{M}}Q(\sqrt{M}f_{\epsilon},\sqrt{M}f_{\epsilon}),\\
&G_3(f_{\epsilon})=e^{-\epsilon\phi_{\epsilon}}+\epsilon\phi_{\epsilon}-1.
\eln\right.
\eq

The linearized collision operator $L$ can be written as
\bq
\left\{\bln
(Lf)(v)&=(Kf)(v)-\nu(v)f(v),\\
\nu(v)&=\int_{\R^3}\int_{\S^2}B(|v-v_{\ast}|,\omega)M_{\ast}d\omega dv_{\ast},\\
(Kf)(v)&=\int_{\R^3}\int_{\S^2}B(|v-v_{\ast}|,\omega)(\sqrt{M_{\ast}'}f'+\sqrt{M'}f_{\ast}'-\sqrt{M}f_{\ast})\sqrt{M_{\ast}}d\omega dv_{\ast}\\
 &=\int_{\R^3}k(v,v_{\ast})f(v_{\ast})dv_{\ast},
\eln\right.
\eq
where $\nu(v)$, the collision frequency, is a real function, and $K$ is a self-adjoint compact operator on $L^2(\mathbb{R}^3_v)$ with a real symmetric integral kernel $k(v,v_{\ast})$. In addition, $\nu(v)$ satisfies
\bq\label{nu1}
\nu_0(1+|v|)^{\gamma}\leq\nu(v)\leq\nu_1(1+|v|)^{\gamma},
\eq
with $\gamma=1$ for the hard sphere and $0\leq\gamma<1$ for the hard potential.

The nullspace of the operator $L$, denote by $N_0$, is a subspace spanned by the orthonormal basis $\{\chi_j,j=0,1,2,3,4\}$ with
\bq
\chi_0=\sqrt{M},\quad \chi_j=v_j\sqrt{M}\quad(j=1,2,3),\quad \chi_4=\frac{(|v|^2-3)\sqrt{M}}{\sqrt{6}}.
\eq

Let $L^2(\R^3_v)$ be a Hilbert space of complex-value function $f(v)$ on $\mathbb{R}^3$ with the inner product and the norm
$$
(f,g)=\int_{\mathbb{R}^3}f(v)\overline{g(v)}dv,\quad \|f\|=\(\int_{\mathbb{R}^3}|f(v)|^2dv\)^{\frac{1}{2}}.
$$
Denote the standard macro-micro decomposition as follows
\bq
f=P_0f+P_1f,\quad  P_0f=\sum^4_{i=0}(f,\chi_i)\chi_i.
\eq
From the Boltzmann's H-theorem, $L$ is non-positive and moreover, $L$ is locally coercive in the sense that there is a constant $\mu>0$ such that
\bq\label{Lf}
(Lf,f)\leq-\mu\|P_1f\|^2,\quad f\in D(L),
\eq
where $D(L)$ is the domains of $L$ given by
$$
D(L)=\big\{f\in L^2(\mathbb{R}^3)\big|\nu(v)f\in L^2(\mathbb{R}^3)\big\}.
$$
Without the loss of generality, we assume in this paper that $\nu(0)\geq\nu_0\geq\mu>0$.

This paper aims to prove the convergence and establish the convergence rate of strong solution $(f_{\epsilon},\phi_{\epsilon})$ towards $(u,\phi)$, where $u=n\chi_0+m\cdot v\chi_0+q\chi_4$ and $(n,m,q,\phi)(t,x)$ is the solution of the following incompressible Navier-Stokes-Poisson-Fourier (NSPF) system:
\bma
&\nabla_x\cdot m=0,\quad n+\sqrt{\frac{2}{3}}q-\phi=0,\label{NS1}\\
&\partial_tm-\kappa_0\Delta_{x}m+\nabla_xp=n\nabla_x\phi-\nabla_x\cdot(m\otimes m),\label{NS3}\\
&\partial_t\bigg(q-\sqrt{\frac{2}{3}}n\bigg)-\kappa_1\Delta_xq=\sqrt{\frac{2}{3}}m\cdot\nabla_x\phi-\frac{5}{3}\nabla_x\cdot(qm),\label{NS4}\\
&(I-\Delta_x)\phi=-n,\label{NS5}
\ema
where $p$ is the pressure, and the initial data $(n,m,q)(0)$ are given by
\bq\label{INS1}
\left\{\bal
m(0)=(f_0,v\chi_0)-\Delta_x^{-1}\nabla_x\div_x(f_0,v\chi_0),\\
q(0)-\sqrt{\frac{2}{3}}n(0)=(f_0,\chi_4-\sqrt{\frac{2}{3}}\chi_0),\\
n(0)+(I-\Delta_x)^{-1}n(0)+\sqrt{\frac{2}{3}}q(0)=0.
\ea\right.
\eq
Here, the viscosity coefficients $\kappa_0$, $\kappa_1>0$ are defined by
\bq
\kappa_0=-(L^{-1}P_1(v_1\chi_2),v_1\chi_2),\quad\kappa_1=-(L^{-1}P_1(v_1\chi_4),v_1\chi_4).\label{kapa}
\eq
In general, the convergence is not uniform near $t=0$ because of the appearance of an initial layer. However, we can show that if the initial data $f_0$ satisfies
\bq\label{F01}
\left\{\bln
&f_0(x,v)=n_0(x)\chi_0+m_0(x)\cdot v\chi_0+q_0(x)\chi_4,\\
&\nabla_x\cdot m_0=0,\quad n_0+(I-\Delta_x)^{-1}n_0+\sqrt{\frac{2}{3}}q_0=0,
\eln\right.
\eq
then the uniform convergence is up to $t=0$.

\noindent\textbf{Notations:}\ \ Before state the main results in this paper, we list some notations. Throughout this paper, $C$ denote a generic positive constant.
For any $\alpha=(\alpha_1,\alpha_2,\alpha_3)\in\mathbb{N}^3$ and $\beta=(\beta_1,\beta_2,\beta_3)\in\mathbb{N}^3$, we denote
$$
\partial^{\alpha}_{x}=\partial^{\alpha_1}_{x_1}\partial^{\alpha_2}_{x_2}\partial^{\alpha_3}_{x_3},\quad \partial^{\beta}_{v}=\partial^{\beta_1}_{v_1}\partial^{\beta_2}_{v_2}\partial^{\beta_3}_{v_3}.
$$
The Fourier transform of $f=f(x,v)$ is defined by
$$
\hat{f}(\xi,v)=\mathcal{F}f(x,v)=\frac{1}{(2\pi)^{\frac{3}{2}}}\int_{\mathbb{R}^3}e^{-\mathrm{i}x\cdot\xi}f(x,v)dx,
$$
where and throughout this paper we denote $\mathrm{i}=\sqrt{-1}$.

For any $\xi\in \R^3$, define a weight Hilbert space $L^2_{\xi}(\mathbb{R}^3)$ by
$$
L^2_{\xi}(\mathbb{R}^3)=\Big\{f\in L^2(\mathbb{R}^3)\big|\|f\|_{\xi}=\sqrt{(f,f)_{\xi}}<\infty\Big\}
$$
with the inner product
\be
(f,g)_{\xi}=(f,g)+\frac{1}{1+|\xi|^2}(P_df,P_dg).
\ee

Set the weight function $w_k$ by
\bq\label{hsw}
w_k=w_k(v)=\(1+|v|^2\)^{\frac{k}{2}}
\eq
for the hard sphere model, or
\bq\label{hpw}
w_k=w_k(t,v)=\(1+|v|^2\)^{\frac{k}{2}}e^{\frac{a|v|}{(1+t)^b}},\quad a,b>0
\eq
for the hard potential model.

Set the Sobolev spaces $H^N=\{f\in L^2\(\mathbb{R}^3_x\times\R^3_v\)|\|f\|_{H^N}<\infty\}$ and $H^N_{w_k}=\{f\in L^2(\mathbb{R}^3_x\times\mathbb{R}^3_v)|\|f\|_{H^N_{w_k}}<\infty\}$ equipped with the norms
$$
\|f\|_{H^N}=\sum_{|\alpha|\le N}\|\partial^{\alpha}_{x}f\|_{L^2(\mathbb{R}^3_x\times\R^3_v)},\quad\|f\|_{H^N_{w_k}}=\sum_{|\alpha|+|\beta|\le N}\|w_k\partial^{\alpha}_{x}\partial^{\beta}_{v}f\|_{L^2(\mathbb{R}^3_x\times\mathbb{R}^3_v)}.
$$
Denote the Banach space $L^{\infty}=L^{\infty}_x(L^2_v)$ equipped with the norm
$$
L^{\infty}=L^{\infty} (\R^3_x,L^2(\R^3_v) ),\quad \|f\|_{L^{\infty}}=\sup_{x\in\R^3}\(\int_{\R^3}|f(x,v)|^2dv\)^{\frac{1}{2}}.
$$
For $q\geq1$ and $k\ge 1$, define
\bmas
L^{2,q}=L^2(\mathbb{R}^3_v,L^q(\mathbb{R}^3_x)),&\quad \|f\|_{L^{2,q}}=\bigg(\int_{\mathbb{R}^3}\(\int_{\mathbb{R}^3}|f(x,v)|^qdx\)^{\frac{2}{q}}dv\bigg)^{\frac{1}{2}},\\
W^{k,q}=L^2(\mathbb{R}^3_v,W^{k,q}(\mathbb{R}^3_x)),&\quad \|f\|_{W^{k,q}}=\bigg(\sum_{|\alpha|\leq k}\int_{\R^3}\(\int_{\R^3}\big|\partial_x^{\alpha}f(x,v)\big|^qdx\)^{\frac{2}{q}}dv\bigg)^{\frac{1}{2}}.
\emas
For simplicity, we denote $W^{k,q}_x=W^{k,q}(\R^3_x)$ and $L^2=L^{2,2}$.

Now we are ready to state main results in this paper.

\begin{thm}\label{thm-1}
(1) Let $N\geq4$. For any $\epsilon\in(0,1)$, there exists a small constant $\delta_0>0$ such that if $\|f_0\|_{H^N_{w_1}}+\|f_0\|_{L^{2,1}}\leq\delta_0$, where $w_1=w_1(v)$ is given by \eqref{hsw} for hard sphere model, and $w_1=w_1(t,v)$ is given by \eqref{hpw} with $a>0$, $0<b\leq\frac{1}{4}$ for hard potential model, then the mVPB system \eqref{Pm1}-\eqref{Pm3} admits a unique global solution $(f_{\epsilon},\phi_{\epsilon})=(f_{\epsilon}(t,x,v),\phi_{\epsilon}(t,x))$ satisfying the following time-decay estimate:
\bq\label{thm-1-1}
\|f_{\epsilon}\|_{H^N_{w_1}}+\|\nabla_x\phi_{\epsilon}\|_{H^N_x}\leq C\delta_0(1+t)^{-\frac{3}{4}},
\eq
where $C>0$ is a constant independent of $\epsilon$.

(2) There exists a small constant $\delta_0>0$ such that if $\|f_0\|_{H^N}+\|f_0\|_{L^{2,1}}\leq\delta_0$, then the NSPF system \eqref{NS1}-\eqref{INS1} admits a unique global solution $(n,m,q,\phi)(t,x)\in L^{\infty}_t\(H^N_x\)$. Moreover, $u(t,x,v)=n(t,x)\chi_0+m(t,x)\cdot v\chi_0+q(t,x)\chi_4$ has the following time-decay rate:
$$
\|u(t)\|_{H^N}+\|\nabla_x\phi(t)\|_{H^N_x}\leq C\delta_0(1+t)^{-\frac{3}{4}},
$$
where $C>0$ is a constant.
\end{thm}


\begin{thm}\label{thm-2}
Let $(f_{\epsilon},\phi_{\epsilon})=(f_{\epsilon}(t,x,v),\phi_{\epsilon}(t,x))$ be the global solution to the mVPB system \eqref{Pm1}-\eqref{Pm3}, and let $(n,m,q,\phi)=(n,m,q,\phi)(t,x)$ be the global solution to the NSPF system \eqref{NS1}-\eqref{INS1}. Then, there exists a small constant $\delta_0>0$ such that if $\|f_0\|_{H^5_{w_1}}+\|f_0\|_{L^{2,1}}\leq\delta_0$, where $w_1=w_1(v)$ is given by \eqref{hsw} for hard sphere model, and $w_1=w_1(t,v)$ is given by \eqref{hpw} with $a>0$, $0<b\leq\frac{1}{4}$ for hard potential model, then we have
\bq
\|f_{\epsilon}(t)-u(t)\|_{L^{\infty}}+\|\nabla_x\phi_{\epsilon}-\nabla_x\phi\|_{L^{\infty}_x}\leq C\delta_0\bigg(\epsilon|\ln\epsilon|^2(1+t)^{-\frac{3}{4}}+\(1+\frac{t}{\epsilon}\)^{-1}\bigg),\label{thm-2-1}
\eq
where $u(t,x,v)=n(t,x)\chi_0+m(t,x)\cdot v\chi_0+q(t,x)\chi_4$.

Moreover, if the initial data $f_0$ satisfies \eqref{F01} and $\|f_0\|_{H^5}+\|f_0\|_{L^{2,1}}\leq\delta_0$, then we have
\bq
\|f_{\epsilon}(t)-u(t)\|_{L^{\infty}}+\|\nabla_x\phi_{\epsilon}-\nabla_x\phi\|_{L^{\infty}_x}\leq C\delta_0\epsilon|\ln\epsilon|(1+t)^{-\frac{3}{4}}.\label{thm-2-2}
\eq
\end{thm}
\begin{remark}\label{dl-rmkil}
Define the oscillation part of $f_{\epsilon}$ as
\be
u_{\epsilon}^{osc}(t,x,v)=\sum_{j=-1,1}\mathcal{F}^{-1}\(e^{\frac{-\mathrm{i}|\xi|u_j(|\xi|)t}{\epsilon}-d_j(|\xi|)t}\big(P_0\hat{f}_0,E_j(\xi)\big)_{\xi}E_j(\xi)\),\label{dlxg-uosc}
\ee
where $u_j(|\xi|)$, $d_j(|\xi|)$ and $E_j(\xi)$, $j=-1,1$ are defined by \eqref{sp4} and \eqref{egf1-1}, respectively. Then, under the first assumption of Theorem \ref{thm-2}, we have
$$
\|f_{\epsilon}(t)-u(t)-u_{\epsilon}^{osc}(t)-e^{\frac{tB_{\epsilon}}{\epsilon^2}}P_1f_0\|_{L^{\infty}}\leq C\delta_0\epsilon|\ln\epsilon|(1+t)^{-\frac{3}{4}}.
$$
Hence, $u_{\epsilon}^{osc}(t)$ and $e^{\frac{tB_{\epsilon}}{\epsilon^2}}P_1f_0$ are the essential components for generating the initial layer.
\end{remark}

\begin{remark}
The coefficients $|\ln\epsilon|^2$ and $|\ln\epsilon|$ in Theorem \ref{thm-2} come from the time convolutions with the initial layer. 
Precisely, under the first assumption of Theorem \ref{thm-2},  the terms related to $|\ln\epsilon|^2$ consist of two parts: the first part  associated with $|\ln\epsilon|$ mainly arises from the convolution of $e^{\frac{t}{\eps^2}B_{\eps}}$ and nonlinear term, namely (See \eqref{dlth2-3-1-j})
$$
\int^t_0\(1+\frac{t-s}{\epsilon}\)^{-1}(1+s)^{-\frac{3}{2}}ds\leq C\epsilon|\ln\epsilon|(1+t)^{-1};
$$
the second part  associated with $|\ln\epsilon|^2$  mainly arises from the convolution of $Y(t)$ and the initial layer, such as (See \eqref{dlth2-3-1-7} and \eqref{dlth2-4-7-6})
\bmas
&\quad\int^{t}_0(1+t-s)^{-\frac{3}{4}}\ln\(2+\frac{1}{t-s}\)\(1+\frac{s}{\epsilon}\)^{-1}(1+s)^{-\frac{3}{4}}ds\\
&\leq C(1+t)^{-\frac{3}{4}}\ln\(2+\frac{2}{t}\)\int^{1}_0\(1+\frac{s}{\epsilon}\)^{-1}ds+\cdot\cdot\cdot\leq C\epsilon|\ln\epsilon|^2(1+t)^{-\frac{3}{4}}+\cdot\cdot\cdot.
\emas
However, under the second assumption of Theorem \ref{thm-2}, the initial layer vanishes, so the estimation  only involve the parts associated with $|\ln\epsilon|$ (See \eqref{dlth-jre-i2314}-\eqref{dlth-2-2-5}).
\end{remark}

Before the rest of the introduction, we will briefly present the main ideas and the approach of the analysis in the proof. The convergence rates given in Theorem \ref{thm-2} of diffusion limit of the mVPB system are proved based on the spectral analysis \cite{Li-4} and the ideas inspired by \cite{Bardos-3,Li-1}.  First of all, the solution $f_{\epsilon}=f_{\epsilon}(t,x,v)$ to the mVPB system \eqref{Pm1}-\eqref{Pm3} can be represented by
$$
f_{\epsilon}(t)=e^{\frac{tB_{\epsilon}}{\epsilon^2}}f_0+\int_0^te^{\frac{(t-s)B_{\epsilon}}{\epsilon^2}}\bigg(G_1 +\frac{1}{\epsilon}G_2 +\frac{1}{\epsilon^2}v\sqrt{M}\cdot\nabla_x(I-\Delta_x)^{-1}G_3 \bigg)(s)ds,
$$
and the solution $u(t,x,v)=n(t,x)\chi_0+m(t,x)\cdot v\chi_0+q(t,x)\chi_4$ to the NSPF system \eqref{NS1}-\eqref{INS1} can be represented by
$$
u(t)=V(t)P_0f_0+\int^t_0V(t-s)\big(Z_1(u)+\div_xZ_2(u)\big)(s)ds,
$$
where $V(t)$ is defined by \eqref{lns4j1-1}, and $Z_i$ $(i=1,2)$ are given by \eqref{nlnsz1}-\eqref{nlnsz2}.

Due to the influence of the electric potential, the linear mVPB operator $B_{\epsilon}(\xi)=L-\mathrm{i}\epsilon(v\cdot\xi)-\mathrm{i}\epsilon\frac{v\cdot\xi}{1+|\xi|^2}P_d$ has no scaling property $B_{\epsilon}(\xi)=B_{1}(\eps\xi)$ as the linear  Boltzmann operator. Thus, we will establish a new non-local implicit function theorem to show that there exist five eigenvalue $\lambda_j(|\xi|,\epsilon)$, $j=-1,0,1,2,3$ for $\epsilon|\xi|$ being small and they satisfy (See Lemma \ref{3dmvpbsp9}):
\bq\label{bjsp-1}
\lambda_j(|\xi|,\epsilon)=-\mathrm{i}\epsilon |\xi|u_j(|\xi|)-\epsilon^2d_j(|\xi|)+O(\epsilon^3|\xi|^3),\quad \epsilon|\xi|\leq r_0,
\eq
where $u_j(|\xi|)$ and $d_j(|\xi|)$ are defined by \eqref{sp4}.

Moreover, by applying the expansion \eqref{bjsp-1} to Lemma \ref{3dmvpbsp12}, we can rewrite $e^{\frac{tB_{\epsilon}(\xi)}{\epsilon^2}}$ as
$$
e^{\frac{tB_{\epsilon}(\xi)}{\epsilon^2}}\hat{f}_0=\sum^3_{j=-1}e^{\frac{-\mathrm{i}|\xi|u_j(|\xi|)t}{\epsilon}-d_j(|\xi|)t+O(\epsilon|\xi|^3)t}\(\big(P_0\hat{f}_0,E_j(\xi)\big)_{\xi}E _j(\xi)+O(\epsilon|\xi|)\)+S_2(t,\xi,\epsilon)\hat{f}_0,
$$
where $E_j(\xi)$ $(j=-1,0,1,2,3)$ are defined by \eqref{egf1-1}, and $S_2(t,\xi,\epsilon)$ is given in Lemma \ref{3dmvpbsp12} and satisfies $\|S_2(t,\xi,\epsilon)\hat{f}_0\|_{\xi}\le Ce^{-\frac{\sigma_0t}{\eps^2}}\|\hat{f}_0\|_{\xi}$. Thus, by using the following key estimate (See Lemma \ref{3dmvpbfas4}):
$$
\bigg\|\mathcal{F}^{-1}\(e^{\frac{-\mathrm{i}|\xi|u_{j}(|\xi|)t}{\epsilon}}\alpha(\omega)(1+|\xi|)^{-3}\)\bigg\|_{L^{\infty}_x}\leq C\(\frac{t}{\epsilon}\)^{-1}, \quad j=\pm1,
$$
where $\alpha(\omega)$ is a smooth function for $\omega=\frac{\xi}{|\xi|}\in\S^2$, we can establish the optimal convergence rate of the semigroup $e^{\frac{tB_{\epsilon}}{\epsilon^2}}$ to its first and second order fluid limits in  $L^{\infty}$ norm as given in Lemma \ref{3dmvpbfas6}.

By using the estimates on the convergence rates for the fluid limits of the linear mVPB system, we can  establish the optimal convergence rate of the strong solution $(f_{\epsilon},\phi_{\epsilon})=(f_{\epsilon}(t,x,v),\phi_{\epsilon}(t,x))$ to the nonlinear mVPB system towards the solution $(u,\phi)=(u(t,x,v),\phi(t,x))$ to the NSPF system. Hence, we obtain the precise estimation on the initial layer.

The rest of this paper will be organized as follows. In Section 2, we will present the results about the spectrum analysis of the linear operator related to the linearized mVPB system. In Section 3, we will establish the first and second order fluid approximations of the solution to the linearized mVPB system. In Section 4, we will prove the convergence and establish the convergence rate of the global solution to the original nonlinear mVPB system to the solution to the nonlinear NSPF system.


\section{Spectral analysis}
\label{sect2}
\setcounter{equation}{0}

In this section, we are concerned with the spectral analysis of the linear mVPB operator $B_{\epsilon}(\xi)$ defined by \eqref{bep1}, which will be applied to study diffusion limit of the solution to the mVPB system \eqref{Pm1}-\eqref{Pm3}.

From \eqref{Pm1}-\eqref{Pm3}, we have the following linearized mVPB system:
\bq\label{Bm1}
\left\{\bal
\epsilon^2\partial_tf_{\epsilon}=B_{\epsilon}f_{\epsilon},\quad t>0,\\
f_{\epsilon}(0,x,v)=f_{0}(x,v),
\ea\right.
\eq
where
\bmas
&B_{\epsilon}f=Lf-\epsilon v\cdot\nabla_xf-\epsilon v\cdot\nabla_x(I-\Delta_x)^{-1}P_df,\\
&P_df=(f,\chi_0)\chi_0, \quad \forall f\in L^2\big(\R^3_v\big).
\emas

Taking Fourier transform to \eqref{Bm1} to get
\bq
\left\{\bal
\epsilon^2\partial_t\hat{f}_{\epsilon}=B_{\epsilon}(\xi)\hat{f}_{\epsilon},\quad t>0,\\
\hat{f}_{\epsilon}(0,\xi,v)=\hat{f}_{0}(\xi,v),
\ea\right.
\eq
where
\bq
B_{\epsilon}(\xi)=L-\mathrm{i}\epsilon(v\cdot\xi)-\mathrm{i}\epsilon\frac{v\cdot\xi}{1+|\xi|^2}P_d.\label{bep1}
\eq

We can regard $B_{\epsilon}(\xi)$ as a linear operator from the space $L^2_{\xi}(\mathbb{R}^3)$ to itself because
\bq\label{fs1}
\|f\|^2\leq\|f\|^2_{\xi}\leq2\|f\|^2.
\eq

We denote $\rho(A)$ and $\sigma(A)$ to be resolvent set and spectrum set of the operator $A$, respectively. The essential spectrum of $A$, denoted by $\sigma_{ess}(A)$, is the set
$\{\lambda\in \mathbb{C} \,|\, \lambda-A~{\rm is~not~a~Fredholm~operator}\}$ (cf. \cite{Kato-1}). The discrete spectrum
of $A$, denoted by $\sigma_d(A)$, is the set $\sigma(A)\setminus \sigma_{ess} (A)$ which consists of  all isolated eigenvalues with finite multiplicity.

\begin{lem}\label{3dmvpbsp1}
The operator $B_{\epsilon}(\xi)$ generates a strongly continuous contraction semigroup on $L^2_{\xi}(\mathbb{R}^3)$, which satisfies
\bq
\|e^{tB_{\epsilon}(\xi)}f\|_{\xi}\leq\|f\|_{\xi},\quad \forall t>0, \  f\in L^2_{\xi}(\mathbb{R}^3).
\eq
\end{lem}
\begin{proof}
Since $P_d$ is self-adjoint projection operator, it follows that $(P_df,P_dg)=(P_df,g)=(f,P_dg)$ and hence
\bq\label{InP1}
(f,g)_{\xi}=\(f,g+\frac{1}{1+|\xi|^2}P_dg\)=\(f+\frac{1}{1+|\xi|^2}P_df,g\).
\eq
By \eqref{InP1}, we have for any $f,g\in L^2_{\xi}(\R^3_v)$ that
\be
(B_{\epsilon}(\xi)f,g)_{\xi}=(f,B_{\epsilon}(\xi)^{\ast}g)_{\xi},
\ee
where  $B_{\epsilon}(\xi)^{\ast}=B_{\epsilon}(-\xi).$  Direct computation gives rise to the
dissipation of both $B_{\epsilon}(\xi)$ and $B_{\epsilon}(\xi)^{\ast}$, namely
$$
\mathrm{Re}(B_{\epsilon}(\xi)f,f)_{\xi}=\mathrm{Re}(B_{\epsilon}(\xi)^{\ast}f,f)_{\xi}=(Lf,f)\leq0.
$$
Since $B_{\epsilon}(\xi)$ is a densely defined closed operator, it follows from Corollary 4.4 on p.15 of \cite{Pazy-1} that $B_{\epsilon}(\xi)$ generates a strongly continuous contraction semigroup on $L^2_{\xi}(\R^3_v)$.
\end{proof}
\begin{lem}\label{3dmvpbsp2}
The following conditions hold for all $\xi\in\R^3$ and $\epsilon\in(0,1)$.

(1) $\sigma_{ess}(B_{\epsilon}(\xi))\subset\{\lambda\in\mathbb{C}\,|\,\mathrm{Re}\lambda\leq-\nu_0\}$ and $\sigma(B_{\epsilon}(\xi))\cap\{\lambda\in\mathbb{C}\,|\,-\nu_0<\mathrm{Re}\lambda\leq0\}\subset\sigma_d(B_{\epsilon}(\xi))$.

(2) If $\lambda$ is an eigenvalue of $B_{\epsilon}(\xi)$, then $\mathrm{Re}\lambda<0$ for any $\xi\neq0 $, and $\lambda(\xi)=0$ if and only if $\xi=0$.
\end{lem}
\begin{proof}
Set
\be
B_{\epsilon}(\xi)=c_{\epsilon}(\xi)+K-\mathrm{i}\epsilon\frac{v\cdot\xi}{1+|\xi|^2}P_d,\quad     c_{\epsilon}(\xi)=-\nu(v)-\mathrm{i}\epsilon(v\cdot\xi).
\ee
By \eqref{nu1}, $\lambda-c_{\epsilon}(\xi)$ is invertible for $\mathrm{Re}\lambda>-\nu_0$ and hence $\sigma(c_{\epsilon}(\xi))\subset\{\lambda\in\mathbb{C}|\mathrm{Re}\lambda\leq-\nu_0\}$. Since $B_{\epsilon}(\xi)$ is a compact perturbation of $c_{\epsilon}(\xi)$, it follows from Theorem 5.35 in p.244 of \cite{Kato-1} that $\sigma_{ess}(B_{\epsilon}(\xi))=\sigma_{ess}(c_{\epsilon}(\xi))$ and $\sigma(B_{\epsilon}(\xi))$ in the domain $\mathrm{Re}\lambda>-\nu_0$ consists of discrete eigenvalues with possible accumulation points only on the line $\mathrm{Re}\lambda=-\nu_0$. This proves (1).

By a similar argument as Lemma 2.2 in \cite{Li-2}, we can prove (2).
\end{proof}

\begin{lem}\label{3dmvpbsp6}
For any fixed $\epsilon\in(0,1)$, the following facts hold.

(1) For $\xi\in\mathbb{R}^3$, there exists $y_1>0$ such that
\bq\label{reso13}
\rho(B_{\epsilon}(\xi))\supset\Big\{\lambda\in\mathbb{C} \,|\,\mathrm{Re}\lambda\geq-\frac{\mu}{2},|\mathrm{Im}\lambda|\geq y_1\Big\}\cup\{\lambda\in\mathbb{C}\,|\,\mathrm{Re}\lambda>0\}.
\eq

(2) For any $r_0>0$, there exists $\alpha=\alpha(r_0)>0$ such that for $\epsilon|\xi|\geq r_0$,
\bq\label{reso13j1}
\sigma(B_{\epsilon}(\xi))\subset\{\lambda\in\mathbb{C} \,|\,\mathrm{Re}\lambda<-\alpha\}.
\eq

(3) For any $\delta>0$, there exists $r_1=r_1(\delta)>0$ such that for $\epsilon|\xi|\leq r_1$,
\bq\label{reso14}
\sigma(B_{\epsilon}(\xi))\cap\Big\{\lambda\in\mathbb{C} \,|\,\mathrm{Re}\lambda\geq-\frac{\mu}{2}\Big\}\subset\big\{\lambda\in\mathbb{C} \,|\,|\lambda|\leq\delta\big\}.
\eq
\end{lem}
\begin{proof}
The proof is similar to Lemma 3.15 in \cite{Li-4}, we omit the detail for brevity.
\end{proof}


Now we study the asymptotic expansions of the eigenvalues and eigenfunctions of $B_{\epsilon}(\xi)$ for $\epsilon|\xi|$ sufficiently small. Let $\tilde{B}_{\epsilon}(\eta)=B_{\epsilon}(\eta e_1)$ with $e_1=(1,0,0)$, we consider the following 1-D eigenvalue problem:
\bq\label{1msp1}
\(L-\mathrm{i}\epsilon v_1\eta-\mathrm{i}\epsilon\frac{v_1\eta}{1+\eta^2}P_d\)e=\beta e,\quad \eta\in\mathbb{R}.
\eq
We have the expansion of the eigenvalues $\beta_j(\eta,\epsilon)$ and the corresponding eigenfunctions $e_j(\eta,\epsilon)$ of $\tilde{B}_{\epsilon}(\eta)$ for $\epsilon|\eta|$  small as follows.

\begin{lem}\label{3dmvpbsp9}
(1) There exists a constant $r_0>0$ such that the spectrum $\sigma(\tilde{B}_{\epsilon}(\eta))\cap\{\lambda\in\mathbb{C}\,|\,\mathrm{Re}\lambda\geq-\frac{\mu}{2}\}$ consists of five points $\{\beta_j(\eta,\epsilon),\,j=-1,0,1,2,3\}$ for $\epsilon|\eta|\leq r_0$. The eigenvalues $\beta_j(\eta,\epsilon)$ and the corresponding eigenfunctions $e_j(\eta,\epsilon)$ are $C^{\infty}$ functions of $\eta$ and $\epsilon$. In particular, the eigenvalues $\beta_j(\eta,\epsilon)$, $j=-1,0,1,2,3$ admit the following asymptotic expansion:
\bq\label{sp6}
\beta_j(\eta,\epsilon)=-\mathrm{i}\epsilon\eta u_j(\eta)-\epsilon^2d_j(\eta)+O(\epsilon^3\eta^3),\quad \epsilon|\eta|\leq r_0,
\eq
where
\bq\label{sp4}
\left\{\bal
u_{\pm1}(\eta)=\mp\sqrt{\frac{5}{3}+\frac{1}{1+\eta^2}},\quad u_j(\eta)=0,~~ j=0,2,3,\\
d_{0}(\eta)=-\frac{\eta^2(3\eta^2+6)}{5\eta^2+8}(L^{-1}P_1(v_1\chi_4),v_1\chi_4),\\
d_{\pm1}(\eta)=-\frac{\eta^2}{2}(L^{-1}P_1(v_1\chi_1),v_1\chi_1)-\frac{\eta^2(\eta^2+1)}{5\eta^2+8}(L^{-1}P_1(v_1\chi_4),v_1\chi_4),\\
d_k(\eta)=-\eta^2(L^{-1}P_1(v_1\chi_2),v_1\chi_2),\quad k=2,3.
\ea\right.
\eq

(2) The corresponding eigenfunctions $e_j(\eta,\epsilon)=e_j(\eta,\epsilon,v)$, $j=-1,0,1,2,3$ satisfy
\bq\label{sp5}
\left\{\bal
(e_j,\overline{e_k})_{\eta}=(e_j,\overline{e_k})+\frac{1}{1+\eta^2}\(P_de_j,P_d\overline{e_k}\)=\delta_{jk},\quad -1\leq j,k\leq3,\\
e_j(\eta,\epsilon)=P_0e_j(\eta,\epsilon)+P_1e_j(\eta,\epsilon),\\
P_0e_j(\eta,\epsilon)=F_j(\eta)+O(\epsilon\eta),\\
P_1e_j(\eta,\epsilon)=\mathrm{i}\epsilon\eta L^{-1}P_1\(v_1F_j(\eta)\)+O(\epsilon^2\eta^2),
\ea\right.
\eq
where $F_j(\eta)\in N_0$, $j=-1,0,1,2,3$ are defined by
\bq
\left\{\bal
F_0(\eta)=\frac{\sqrt{2}(1+\eta^2)}{\sqrt{(\eta^2+2)(5\eta^2+8)}}\chi_0-\frac{\sqrt{3\eta^2+6}}{\sqrt{5\eta^2+8}}\chi_4,\\
F_{\pm1}(\eta)=\frac{\sqrt{3\eta^2+3}}{\sqrt{10\eta^2+16}}\chi_0\mp\frac{\sqrt{2}}{2}\chi_1+\frac{\sqrt{\eta^2+1}}{\sqrt{5\eta^2+8}}\chi_4,\\
F_k(\eta)=\chi_k,\quad k=2,3.
\ea\right.
\eq
\end{lem}
\begin{proof}
Let $e$ be the eigenfunction of \eqref{1msp1} and $\beta=-\mathrm{i}\epsilon\eta\sigma$, we rewrite $e$ in the form $e=g_0+g_1$, where $g_0=P_0e$ and $g_1=P_1e$. The eigenvalue problem \eqref{1msp1} can be decomposed into
\bma
-\mathrm{i}\epsilon\eta\sigma g_0&=-\mathrm{i}\epsilon\eta P_0[v_1(g_0+g_1)]-\mathrm{i}\epsilon\frac{v_1\eta}{1+\eta^2}P_dg_0,\label{1msp2}\\
-\mathrm{i}\epsilon\eta\sigma g_1&=Lg_1-\mathrm{i}\epsilon \eta P_1[v_1(g_0+g_1)].\label{1msp3}
\ema
From \eqref{1msp3}, we obtain that for any $\mathrm{Re}(-\mathrm{i}\epsilon\eta\sigma)>-\frac{\mu}{2}$,
\bq\label{1msp4}
g_1=\mathrm{i}\epsilon \eta(L+\mathrm{i}\eta\epsilon\sigma -\mathrm{i}\epsilon \eta P_1v_1P_1)^{-1}(P_1v_1g_0).
\eq
Substituting \eqref{1msp4} into \eqref{1msp2}, we have
\bq\label{1msp5}
\sigma g_0= P_0v_1g_0+\frac{v_1 }{1+\eta^2}P_dg_0+ \mathrm{i}\epsilon \eta P_0\(v_1R(\sigma,\epsilon\eta)P_1v_1g_0\),
\eq
where
$$
R(\sigma, \epsilon\eta)=(L+\mathrm{i}\epsilon\eta\sigma -\mathrm{i} \epsilon\eta P_1v_1P_1)^{-1}.
$$

Define the operator $A(\eta)=P_0v_1P_0+\frac{v_1}{1+\eta^2}P_d$. We have the matrix representation of $A(\eta)$ as
\bq
A(\eta)=\left(
  \begin{array}{ccccc}
    0 & 1 & 0 & 0 & 0 \\
   1+\frac{1}{1+\eta^2} & 0 & 0 & 0 & \sqrt{\frac{2}{3}} \\
    0 & 0 & 0 & 0 & 0 \\
    0 & 0 & 0 & 0 & 0 \\
    0 & \sqrt{\frac{2}{3}} & 0 & 0 & 0 \\
  \end{array}
\right).
\eq
It can be verified that the eigenvalues $u_j(\eta)$ and  eigenvectors $F_j(\eta)$ of $A(\eta)$ are given by
\bq\label{dlEUJ}
\left\{\bal
u_{\pm1}(\eta)=\mp\sqrt{\frac{5}{3}+\frac{1}{1+\eta^2}},\quad u_j(\eta)=0, \quad j=0,2,3,\\
F_{\pm1}(\eta)=\frac{\sqrt{3\eta^2+3}}{\sqrt{10\eta^2+16}}\chi_0\mp\frac{\sqrt{2}}{2}\chi_1+\frac{\sqrt{\eta^2+1}}{\sqrt{5\eta^2+8}}\chi_4,\\
F_0(\eta)=\frac{\sqrt{2}(1+\eta^2)}{\sqrt{(\eta^2+2)(5\eta^2+8)}}\chi_0-\frac{\sqrt{3\eta^2+6}}{\sqrt{5\eta^2+8}}\chi_4,\\
F_k(\eta)=\chi_k,\quad k=2,3,\\
(F_j(\eta),F_k(\eta))_{\eta}=\delta_{jk},\quad -1\leq j,k\leq3.
\ea\right.
\eq

We will  reduce \eqref{1msp5} to a problem of 5-dimensional  system. Since $g_0\in N_0$, we can represent $g_0$ as
$$
g_0=\sum^{4}_{j=0}C_jF_{j-1}(\eta)\quad \mathrm{with}\quad C_j=(g_0,F_{j-1}(\eta))_{\eta},\quad j=0,1,2,3,4.
$$

Taking the inner product $(\cdot,\cdot)_{\eta}$ between \eqref{1msp5} and $F_j(\eta)$~$(j=-1,0,1,2,3)$ respectively, we obtain the equations about $\sigma$ and $(C_0,C_1,C_2,C_3,C_4)$:
\be
\sigma C_j =u_{j-1}(\eta)C_j +\mathrm{i}\epsilon\eta \sum^4_{k=0}C_k R_{kj}(\sigma,\eta,\epsilon), \quad j=0,1,2,3,4,\label{1msp5j}
\ee
where
\bq\label{1msp7}
R_{jk}(\sigma,\eta,\epsilon)=\((L+\mathrm{i}\epsilon\eta\sigma-\mathrm{i} \epsilon\eta P_1v_1P_1)^{-1}P_1v_1F_{j-1}(\eta),v_1F_{k-1}(\eta)\).
\eq
In particular, $R_{jk}(\sigma,\eta,\epsilon)$, $j,k=0,1,2,3,4$ satisfy
\bq\label{RJJ}
\left\{\bln
&R_{kl}(\sigma,\eta,\epsilon)=R_{lk}(\sigma,\eta,\epsilon)=0,\quad k=0,1,2,\quad l=3,4,\\
&R_{34}(\sigma,\eta,\epsilon)=R_{43}(\sigma,\eta,\epsilon)=0,\\
&R_{33}(\sigma,\eta,\epsilon)=R_{44}(\sigma,\eta,\epsilon).
\eln\right.
\eq

Since $F_2,F_3$ are independent of $\eta$, we can rewrite $R_{kk}(\sigma,\eta,\epsilon)$ as $R_{kk}(\sigma,\epsilon\eta)$ for $k=3,4$. Thus, by \eqref{1msp7} and \eqref{RJJ}, we can divide \eqref{1msp5j} into the following two systems:
\bma
\sigma C_j&=u_{j-1}(\eta)C_j+\mathrm{i}\epsilon\eta\sum^{2}_{i=0}C_iR_{ij}(\sigma,\eta,\epsilon),\quad j=0,1,2,\label{1msp5j1}\\
\sigma C_k&=\mathrm{i}\epsilon\eta C_kR_{33}(\sigma,\epsilon\eta),\quad k=3,4.\label{1msp5j2}
\ema

Denote
\bma
D_0(\sigma,z)&=\sigma-\mathrm{i}z R_{33}(\sigma,z),\label{1msp9}\\
D_1(\sigma,\eta,\epsilon)&=
 \left|
\begin{array}{ccc}
\sigma-u_{-1}-\mathrm{i}\epsilon\eta R_{00}&-\mathrm{i}\epsilon\eta R_{10}&-\mathrm{i}\epsilon\eta R_{20}\\
-\mathrm{i}\epsilon\eta R_{01}&\sigma-u_0-\mathrm{i}\epsilon\eta R_{11}&-\mathrm{i}\epsilon\eta R_{21}\\
-\mathrm{i}\epsilon\eta R_{02}&-\mathrm{i}\epsilon\eta R_{12}&\sigma-u_1-\mathrm{i}\epsilon\eta R_{22}\\
\end{array}
\right|,\label{1msp10}
\ema
where $z=\epsilon\eta$ and $R_{jk}=R_{jk}(\sigma,\eta,\epsilon)$ $(j,k=0,1,2)$ is defined by \eqref{1msp7}.


We have the following result about the solutions of $D_0(\sigma, z)=0$ and $D_1(\sigma, \eta,\epsilon)=0$.

\begin{lem}\label{3dmvpbsp7}

(1) There are two constants $r_0,~r_1>0$ such that the equation $D_0(\sigma, z)=0$ has a unique  solution $\sigma=\sigma(z)$ for $(z,\sigma)\in[-r_0,r_0]\times B_{r_1}(0)$, which is a $C^{\infty}$ function of $z$ satisfying
\bq
\sigma(0)=0,\quad  \partial_{z}\sigma(0)=\mathrm{i}(L^{-1}P_1(v_1\chi_2),v_1\chi_2).
\eq

(2) There are two small constants $r_0,~r_1>0$ such that the equation $D_1(\sigma,\eta,\epsilon)=0$ has exactly three solutions $\sigma_j=\sigma_j(\eta,\epsilon)$, $j=-1,0,1$ for $\epsilon|\eta|\leq r_0$ and $|\sigma_j-u_j(\eta)|\leq r_1|\eta|$. They are $C^{\infty}$ functions of $\eta$ and $\epsilon$, which satisfies
\bq\label{sp1}
\sigma_j(\eta,0)=u_j(\eta),\quad \partial_{\epsilon}\sigma_j(\eta,0)=-b_j(\eta),
\eq
where $u_j,b_j$, $j=-1,0,1$ are defined by \eqref{sp4}.

In particular, $\sigma_j(\eta,\epsilon)$, $j=-1,0,1$ satisfy the following expansions
\bq\label{sp3}
\sigma_j(\eta,\epsilon)=u_j(\eta)-\epsilon b_j(\eta)+O(\epsilon^2\eta^2).
\eq
\end{lem}
\begin{proof}
By direct computation and the implicit function theorem, we can prove (1).

Now, we want to show (2). By \eqref{1msp10}, we have
\bmas
D_1(\sigma,\eta,\epsilon)
&=(\sigma-u_{-1})(\sigma-u_{0})(\sigma-u_{1})-\mathrm{i}\eps \eta \sum_{jkl,l\ge k} R_{jj} (\sigma-u_{k-1}) (\sigma-u_{l-1})
\\
&\quad-\eps^2\eta^2 \sum_{jkl,j\ge k}(R_{jj}R_{kk}-R_{jk}R_{kj})(\sigma-u_{l-1})+\mathrm{i}\eps^3\eta^3 \sum_{jkl}(-1)^{\tau(jkl)}R_{0j}R_{1k}R_{2l},
\emas
where  $jkl$ denotes the arrange of $j,k,l=0,1,2$, $\sum_{jkl}$ denotes the summation of all arrange $jkl$, $\tau(jkl)$ denotes the inversion number of $jkl$.
It follows that
\bq\label{1msp10j}
D_1(\sigma,\eta,0)=(\sigma-u_{-1})(\sigma-u_{0})(\sigma-u_{1}).
\eq
This implies that $D_1(\sigma,\eta,0)=0$ has three roots $u_j(\eta)$ for $j=-1,0,1$.  Moreover, $D_1(\sigma,\eta,\epsilon)$ is $C^{\infty}$ with respect to $(\sigma,\eta,\epsilon)$ and satisfies
\bma
\partial_{\epsilon}D_1(\sigma,\eta,\epsilon)&=-\mathrm{i} \eta \sum_{jkl,l\ge k} (R_{jj}+\eps\pt_{\eps}R_{jj}) (\sigma-u_{k-1}) (\sigma-u_{l-1})\nnm\\
&\quad-\eps \eta^2 \sum_{jkl,j\ge k}(2R_{jj}R_{kk}+\eps \pt_{\eps}(R_{jj}R_{kk})-R_{jk}R_{kj}-\eps \pt_{\eps}(R_{jk}R_{kj}))(\sigma-u_{l-1}) \nnm\\
&\quad+\mathrm{i}\epsilon^2\eta^3\sum_{jkl}(-1)^{\tau(jkl)}(3R_{1j}R_{2k}R_{3l}+\eps\pt_{\eps} (R_{1j}R_{2k}R_{3l})),\label{1msp11}
\\
\partial_{\sigma}D_1(\sigma,\eta,\epsilon)&=3\sigma^2-\frac{5}{3}-\frac{1}{1+\eta^2}-\mathrm{i}\eps \eta \sum_{jkl,l\ge k} \pt_{\sigma}R_{jj} (\sigma-u_{k-1}) (\sigma-u_{l-1})\nnm\\
&\quad-\mathrm{i}\eps \eta \sum_{jkl,l\ge k}R_{jj}(2\sigma-u_{k-1}-u_{l-1})-\eps^2\eta^2 \sum_{jkl,j\ge k}\pt_{\sigma}(R_{jj}R_{kk}-R_{jk}R_{kj})(\sigma-u_{l-1})\nnm\\
&\quad-\eps^2\eta^2 \sum_{j\ge k}(R_{jj}R_{kk}-R_{jk}R_{kj}) +\mathrm{i}\eps^3\eta^3 \sum_{jkl}(-1)^{\tau(jkl)}\pt_{\sigma}(R_{0j}R_{1k}R_{2l}).\label{1msp12}
\ema

For $j=-1,0,1$, we define
$$
\Pi_j(\sigma,\eta,\epsilon)=\sigma-\(3u_j(\eta)^2-\frac{5}{3}-\frac{1}{1+\eta^2}\)^{-1}D_1(\sigma,\eta,\epsilon).
$$
It is straightforward to verify that a solution of $D_1(\sigma,\eta,\epsilon)=0$ for any fixed $\eta$ and $\epsilon$ is a fixed point of $\Pi_j(\sigma,\eta,\epsilon)$.

Since
\bq\label{spRR}
|\partial_{\sigma}R_{jk}(\sigma,\eta,\epsilon)|\leq C|\eta|\epsilon,\quad |\partial_{\epsilon}R_{jk}(\sigma,\eta,\epsilon)|\leq C|\eta|(|\sigma|+1),\quad j,k=0,1,2,
\eq
it follows from \eqref{1msp11} and \eqref{1msp12} that
\bma
&|\partial_{\sigma}\Pi_j(\sigma,\eta,\epsilon)|=\bigg|1-\(3u_j(\eta)^2-\frac{5}{3}-\frac{1}{1+\eta^2}\)^{-1}\partial_{\sigma}D_1(\sigma,\eta,\epsilon)\bigg|\leq Cr_1,\nonumber\\
&|\partial_{\epsilon}\Pi_j(\sigma,\eta,\epsilon)|=\bigg|\(3u_j(\eta)^2-\frac{5}{3}-\frac{1}{1+\eta^2}\)^{-1}\partial_{\epsilon}D_1(\sigma,\eta,\epsilon)\bigg|\leq C|\eta|,\nonumber
\ema
for $|\sigma-u_j(\eta)|\leq r_1|\eta|$ and $\epsilon|\eta|\leq r_0$ with $r_0,r_1>0$ sufficiently small. This implies that for $|\sigma-u_j(\eta)|\leq r_1|\eta|$ and $\epsilon|\eta|\leq r_0$ with $r_0,r_1\ll1$,
\bma
|\Pi_j(\sigma,\eta,\epsilon)-u_j(\eta)|&=|\Pi_j(\sigma,\eta,\epsilon)-\Pi_j(u_j(\eta),\eta,0)|\nonumber\\
&\leq |\Pi_j(\sigma,\eta,\epsilon)-\Pi_j(\sigma,\eta,0)|+|\Pi_j(\sigma,\eta,0)-\Pi_j(u_j(\eta),\eta,0)|\nonumber\\
&\leq |\partial_{\epsilon}\Pi_j(\sigma,\eta,\tilde{\epsilon})||\epsilon|+|\partial_{\sigma}\Pi_j(\tilde{\sigma},\eta,0)||\sigma-u_j(\eta)|\leq r_1|\eta|,\nonumber\\
|\Pi_j(\sigma_1,\eta,\epsilon)-\Pi_j(\sigma_2,\eta,\epsilon)|&\leq |\partial_{\sigma}\Pi_j(\bar{\sigma},\eta,\epsilon)||\sigma_1-\sigma_2|\leq \frac{1}{2}|\sigma_1-\sigma_2|,\nonumber
\ema
where $\tilde{\epsilon}$ is between $0$ and $\epsilon$, $\tilde{\sigma}$ is between $\sigma$ and $u_{j}(\eta)$, and $\bar{\sigma}$ is between $\sigma_1$ and $\sigma_2$.

Hence by the contraction mapping theorem, there exist exactly three functions $\sigma_j(\eta,\epsilon)$, $j=-1,0,1$ for $\epsilon|\eta|\leq r_0$ and $|\sigma_j-u_j(\eta)|\leq r_1|\eta|$ such that
$\Pi_j(\sigma_j(\eta,\epsilon),\eta,\epsilon)=\sigma_j(\eta,\epsilon)$ and $\sigma_j(\eta,0)=u_j(\eta)$. This is equivalent to that $D_1(\sigma_j(\eta,\epsilon),\eta,\epsilon)=0$.
Moreover, by \eqref{1msp11}-\eqref{1msp12} we have
\bma
\partial_{\epsilon}\sigma_0(\eta,0)&=-\frac{\partial_{\epsilon}D_1(0,\eta,0)}{\partial_{\sigma}D_1(0,\eta,0)}=\frac{\mathrm{i}\eta(3\eta^2+6)}{(5\eta^2+8)}(L^{-1}P_1v_1\chi_4,v_1\chi_4),\label{1msp13}\\
\partial_{\epsilon}\sigma_{\pm1}(\eta,0)&=-\frac{\partial_{\epsilon}D_1(u_{\pm1}(\eta),\eta,0)}{\partial_{\sigma}D_1(u_{\pm1}(\eta),\eta,0)}\nonumber\\
&=\frac{\mathrm{i}\eta}{2}(L^{-1}P_1v_1\chi_1,v_1\chi_1)+\frac{\mathrm{i}\eta(\eta^2+1)}{5\eta^2+8}(L^{-1}P_1v_1\chi_4,v_1\chi_4).\label{1msp15}
\ema
Combining \eqref{1msp10j}, \eqref{1msp13} and \eqref{1msp15}, we obtain \eqref{sp1}.

Finally, we deal with \eqref{sp3}. Since
\bmas
|\sigma_j(\eta,\epsilon)-u_j(\eta)|&\leq|\Pi_j(\sigma_j(\eta,\epsilon),\eta,\epsilon)-\Pi_j(u_j(\eta),\eta,\epsilon)|\nnm\\
&\quad+|\Pi_j(u_j(\eta),\eta,\epsilon)-\Pi_j(u_j(\eta),\eta,0)|\nnm\\
&\leq|\partial_{\sigma}\Pi_j(\tilde{\sigma},\eta,\epsilon)||\sigma_j(\eta,\epsilon)-u_j(\eta)|+|\partial_{\epsilon}\Pi_j(u_j(\eta),\eta,\tilde{\epsilon})|\epsilon,
\emas
it follows that
$$
|\sigma_j(\eta,\epsilon)-u_j(\eta)|\leq(1-|\partial_{\sigma}\Pi_j(\tilde{\sigma},\eta,\epsilon)|)^{-1}|\partial_{\epsilon}\Pi_j(u_j(\eta),\eta,\tilde{\epsilon})|\epsilon\leq C\epsilon|\eta|,
$$
where $\tilde{\epsilon}$ is between $0$ and $\epsilon$, $\tilde{\sigma}$ is between $\sigma_{j}(\eta,\epsilon)$ and $u_{j}(\eta)$.

By \eqref{1msp11}-\eqref{spRR}, we obtain that for $|\sigma_j(\eta,\epsilon)-u_{j}(\eta)|\leq C\epsilon|\eta|$ and $\epsilon|\eta|\leq r_0$,
\bmas
\partial_{\epsilon}D_1(\sigma_j,\eta,\epsilon)&=-\i\eta R_{00}(u_j-u_0)(u_j-u_{1})-\i\eta R_{11}(u_j-u_{-1})(u_j-u_{1})\\
&\quad-\i\eta R_{22}(u_j-u_{-1})(u_j-u_{0})+O(1)\epsilon|\eta|^2,\\
\partial_{\sigma}D_1(\sigma_j,\eta,\epsilon)&=(u_j-u_0)(u_j-u_{1})+(u_j-u_{-1})(u_j-u_{1})\\
&\quad+(u_j-u_{-1})(u_j-u_{0})+O(1)\epsilon|\eta|,
\emas
which gives
\bmas
&|\partial_{\sigma}\Pi_j(\sigma_j,\eta,\epsilon)|=\bigg|1-\(3u_j(\eta)^2-\frac{5}{3}-\frac{1}{1+\eta^2}\)^{-1}\partial_{\sigma}D_1(\sigma_j,\eta,\epsilon)\bigg|=O(1)\epsilon|\eta|,\\
&|\partial_{\epsilon}\Pi_j(\sigma_j,\eta,\epsilon)|=\bigg|\(3u_j(\eta)^2-\frac{5}{3}-\frac{1}{1+\eta^2}\)^{-1}\partial_{\epsilon}D_1(\sigma_j,\eta,\epsilon)\bigg|=b_j(\eta)+O(1)\epsilon|\eta|^2.
\emas
Thus,
\bmas
&|\sigma_j(\eta,\epsilon)-u_j(\eta)-\epsilon b_j(\eta)|\\
=&\,|\Pi_j(\sigma_j(\eta,\epsilon),\eta,\epsilon)-\Pi_j(u_j(\eta),\eta,\epsilon)|\\
&+|\Pi_j(u_j(\eta),\eta,\epsilon)-\Pi_j(u_j(\eta),\eta,0)-\epsilon b_j(\eta)|\\
\leq&\,|\partial_{\sigma}\Pi_j(\tilde{\sigma},\eta,\epsilon)||\sigma_j-u_j|+|\partial_{\epsilon}\Pi_j(\sigma_j,\eta,\tilde{\epsilon})-b_j(\eta)|\epsilon=O(1)\epsilon^2|\eta|^2.
\emas
This proved \eqref{sp3}.
\end{proof}

The eigenvalues $\beta_j(\eta,\epsilon)$ and the corresponding eigenfunctions $e_j(\eta,\epsilon)$ $(j=-1,0,1,2,3)$ of $\tilde{B}_{\epsilon}(\eta)$ can be constructed as follows. For $j=2,3,$ we take $\beta_j(\eta,\epsilon)=-\mathrm{i}\epsilon\eta\sigma(\epsilon\eta)$ with $\sigma(z)$ being a solution of $D_0(\sigma, z)=0$ defined in Lemma \ref{3dmvpbsp7}. We denote by $ C_2(\epsilon\eta) $ as a solution of system \eqref{1msp5j2} for $\sigma=\sigma(\epsilon\eta)$. The corresponding eigenfunctions $e_j(\eta,\epsilon)$, $j=2,3$ are defined by
\bq\label{1msp16}
e_j(\eta,\epsilon)=C_2(\epsilon\eta)F_j+\mathrm{i}C_2(\epsilon\eta)\epsilon\eta(L+\mathrm{i}\epsilon\eta\sigma_j-\mathrm{i}\epsilon\eta P_1v_1P_1)^{-1}P_1(v_1F_j),
\eq
which are orthonormal, i.e., $ (e_2(\eta,\epsilon),\overline{e_3(\eta,\epsilon)} )_{\eta}=0$.

For $j=-1,0,1$, we choose $\beta_j(\eta,\epsilon)=-\mathrm{i}\eta\epsilon\sigma_j(\eta,\epsilon)$ with $\sigma_j(\eta,\epsilon)$ being a solution of $D_1(\sigma,\eta,\epsilon)=0$ defined in Lemma \ref{3dmvpbsp7}. We denote by $ \{C_0^j(\eta,\epsilon),C_1^j(\eta,\epsilon),C_2^j(\eta,\epsilon) \}$ as a solution of system \eqref{1msp5j1} for $\sigma=\sigma_j(\eta,\epsilon)$. Then we can construct $e_j(\eta,\epsilon)$, $j=-1,0,1$ as
\bq\label{1msp16j}
\left\{\bal
e_j(\eta,\epsilon)=P_0e_j(\eta,\epsilon)+P_1e_j(\eta,\epsilon),\\
P_0e_j(\eta,\epsilon)=C_0^j(\eta,\epsilon)F_{-1}(\eta)+C_1^j(\eta,\epsilon)F_0(\eta)+C_2^j(\eta,\epsilon)F_1(\eta),\\
P_1e_j(\eta,\epsilon)=\mathrm{i}\epsilon\eta(L+\mathrm{i}\eta\epsilon\sigma_j-\mathrm{i}\eta\epsilon P_1v_1P_1)^{-1}P_1(v_1P_0e_j(\eta,\epsilon)).
\ea\right.
\eq
We write
\bq
\(L-\mathrm{i}\eta\epsilon v_1-\mathrm{i}\epsilon\frac{v_1\eta}{1+\eta^2}P_d\)e_j(\eta,\epsilon)=\beta_j(\eta,\epsilon)e_j(\eta,\epsilon),\quad j=-1,0,1,2,3.\label{dl-mvpbsp-1-1}
\eq
By taking inner product $(\cdot,\cdot)_{\eta}$ between \eqref{dl-mvpbsp-1-1} and $\overline{e_k(\eta,\epsilon)}$, and using the fact that
\bma
&\(L+\mathrm{i}\eta\epsilon v_1+\mathrm{i}\epsilon\frac{v_1\eta}{1+\eta^2}P_d\)\overline{e_j(\eta,\epsilon)}=\overline{\beta_j(\eta,\epsilon)}\overline{e_j(\eta,\epsilon)},\nonumber\\
&(\tilde{B}_{\epsilon}(\eta)f,g)_{\eta}=(f,\tilde{B}_{\epsilon}(-\eta)g)_{\eta},\quad \forall f,g\in D(\tilde{B}_{\epsilon}(\eta)),\nonumber
\ema
we have
$$
(\beta_j(\eta,\epsilon)-\beta_k(\eta,\epsilon))(e_j(\eta,\epsilon),\overline{e_k(\eta,\epsilon)})_\eta=0,\quad -1\leq j\neq k\leq3.
$$
For $\epsilon|\eta|>0$ sufficiently small, $\beta_j(\eta,\epsilon)\neq\beta_k(\eta,\epsilon)$ for $-1\leq j\neq k\leq3$. Therefore, we have the orthogonality relation
$$
(e_j(\eta,\epsilon),\overline{e_k(\eta,\epsilon)})_\eta=0,\quad -1\leq j\neq k\leq3.
$$
We also normalized this eigenfunctions by
$$(e_j(\eta,\epsilon),\overline{e_j(\eta,\epsilon)})_\eta=1, \quad -1\leq j\leq3.$$

Let $z=\epsilon\eta$. The coefficient $C_2(z)$  in \eqref{1msp16} is determined by the normalization condition as
\bq\label{1msp17}
C_2(z)^2\big(1+z^2D_2(z)\big)=1,
\eq
with $D_2(z)=\big(R(\sigma,z)P_1(v_1F_2),\overline{R(\sigma,z)}P_1(v_1F_2)\big)$. It follows from \eqref{1msp17} that
$$
C_2(z)=1+O(z^2),\quad |z|\leq r_0,
$$
which together with \eqref{1msp16} leads to \eqref{sp5} for $j=2,3$.

To obtain the expansion of $e_j(\eta,\epsilon)$ for $j=-1,0,1$ defined in \eqref{1msp16j}, we consider its macroscopic part and microscopic part respectively. By \eqref{1msp5j}, the macroscopic part $P_0e_j(\eta,\epsilon)$ is determined in terms of the coefficients $ \{C_0^j(\eta,\epsilon),C_1^j(\eta,\epsilon),C_2^j(\eta,\epsilon) \}$ $(j=-1,0,1)$ that satisfy
\bq\label{1msp18}
\sigma_j(\eta,\epsilon) C_k^j(\eta,\epsilon)=u_{j-1}(\eta)C_k^j(\eta,\epsilon)+\i\epsilon\eta\sum^{2}_{l=0}C_l^j(\eta,\epsilon)R_{lk}(\sigma,\epsilon\eta),\quad k=0,1,2.
\eq
Furthermore, we have the normalization conditions:
\bq\label{1msp19}
1\equiv(e_j(\eta,\epsilon),\overline{e_j(\eta,\epsilon)})_{\eta}=C_0^j(\eta,\epsilon)^2+C_1^j(\eta,\epsilon)^2+C_2^j(\eta,\epsilon)^2+O(\epsilon^2\eta^2),\quad \epsilon|\eta|\leq r_0.
\eq

By \eqref{sp3} and \eqref{1msp18}, we can expand $C_0^j(\eta,\epsilon),C_1^j(\eta,\epsilon),C_2^j(\eta,\epsilon)$ as
$$
C_0^j(\eta,\epsilon)=C_{0,0}^j(\eta)+O(\epsilon\eta),\quad C_1^j(\eta,\epsilon)=C_{1,0}^j(\eta)+O(\epsilon\eta),\quad C_2^j(\eta,\epsilon)=C_{2,0}^j(\eta)+O(\epsilon\eta).
$$
Substituting above expansions and \eqref{sp3} into \eqref{1msp18} and \eqref{1msp19}, we have
\bq\label{1msp20}
\left\{\bal
u_j(\eta)C_{k,0}^j(\eta)=u_{k-1}(\eta)C^j_{k,0}(\eta),\\
C^j_{0,0}(\eta)^2+C^j_{1,0}(\eta)^2+C^j_{2,0}(\eta)^2=1,
\ea\right.
\eq
where $j=-1,0,1$, $k=0,1,2$. By straightforward computation, we can obtain from \eqref{1msp20} that
\bq\label{1msp-21}
C^j_{j+1,0}(\eta)=1,\quad C^j_{k,0}(\eta)=0,\quad k\neq j+1.
\eq
By \eqref{1msp16j} and \eqref{1msp-21}, we obtain the expansion of $e_j(\eta,\epsilon)$~$(j=-1,0,1)$. The proof is then completed.
\end{proof}

We now consider the following 3-D eigenvalue problem:
\bq\label{1msp21}
B_{\epsilon}(\xi)\psi=\(L-\mathrm{i}\epsilon(v\cdot\xi)-\mathrm{i}\epsilon\frac{v\cdot\xi}{1+|\xi|^2}P_d\)\psi=\lambda\psi,\quad \xi\in\mathbb{R}^3.
\eq


With the help of Lemma \ref{3dmvpbsp9}, we have the expansion of the eigenvalues $\lambda(|\xi|,\epsilon)$ and the corresponding eigenfunctions $\psi_j(\xi,\epsilon)$ of $B_{\epsilon}(\xi)$ for $\epsilon|\xi|\leq r_0$ as follows.
\begin{lem}\label{3dmvpbsp10}
(1) There exists a constant $r_0>0$ such that the spectrum $\sigma(B_{\epsilon}(\xi))\cap\{\lambda\in\mathbb{C}|\mathrm{Re}\lambda\geq-\frac{\mu}{2}\}$ consists of five points $\{\lambda_j(|\xi|,\epsilon),j=-1,0,1,2,3\}$ for $\epsilon|\xi|\leq r_0$. The eigenvalues $\lambda_j(|\xi|,\epsilon)$, $j=-1,0,1,2,3$ are $C^{\infty}$ functions of $(|\xi|,\epsilon)$, and satisfy the following expansions for $\epsilon|\xi|\leq r_0$:
\bq
\lambda_j(|\xi|,\epsilon)=-\mathrm{i}\epsilon |\xi|u_j(|\xi|)-\epsilon^2 d_j(|\xi|)+O(\epsilon^3|\xi|^3),\quad \epsilon|\xi|\leq r_0,
\eq
where $u_j(|\xi|)$ and $d_j(|\xi|)$ are defined by  \eqref{sp4}.

(2) The eigenfunctions $\psi_j(\xi,\epsilon)=\psi_j(\xi,\epsilon,v)$, $j=-1,0,1,2,3$ satisfy
\bq\label{egf1}
\left\{\bal
\big(\psi_j,\overline{\psi_k}\big)_{\xi}=\big(\psi_j,\overline{\psi_k}\big)+\frac{1}{1+|\xi|^2}\big(P_d\psi_j,P_d\overline{\psi_k}\big)=\delta_{jk},\quad -1\leq j,k\leq3,\\
\psi_j(\xi,\epsilon)=P_0\psi_j(\xi,\epsilon)+P_1\psi_j(\xi,\epsilon),\\
P_0\psi_j(\xi,\epsilon)=E_j(\xi)+O(\epsilon|\xi|),\\
P_1\psi_j(\xi,\epsilon)=\mathrm{i}\epsilon L^{-1}P_1((v\cdot\xi) E_j(\xi))+O(\epsilon^2|\xi|^2),
\ea\right.
\eq
where $E_j(\xi)\in N_0$ are defined by
\bq\label{egf1-1}
\left\{\bal
E_0(\xi)=\frac{\sqrt{2}(1+|\xi|^2)}{\sqrt{(|\xi|^2+2)(5|\xi|^2+8)}}\chi_0-\frac{\sqrt{3|\xi|^2+6}}{\sqrt{5|\xi|^2+8}}\chi_4,\\ E_{\pm1}(\xi)=\frac{\sqrt{3|\xi|^2+3}}{\sqrt{10|\xi|^2+16}}\chi_0\mp\frac{\sqrt{2}(v\cdot\xi)}{2|\xi|}\chi_0+\frac{\sqrt{|\xi|^2+1}}{\sqrt{5|\xi|^2+8}}\chi_4,\\
E_k(\xi)=v\cdot W^{k}\chi_0,\quad k=2,3,
\ea\right.
\eq
and $W^k$ $(k=2,3)$ are orthonormal vectors satisfying $W^k\cdot\xi=0$.
\end{lem}
\begin{proof}
Let $\mathbb{O}$ be a rotational transformation in $\mathbb{R}^3$ such that $\mathbb{O}:\frac{\xi}{|\xi|}\rightarrow(1,0,0)$. We have
$$
\mathbb{O}^{-1}\(L-\mathrm{i}\epsilon(v\cdot\xi)-\mathrm{i}\epsilon\frac{v\cdot\xi}{1+|\xi|^2}P_d\)\mathbb{O}=L-\mathrm{i}\epsilon|\xi| v_1-\mathrm{i}\epsilon\frac{v_1|\xi|}{1+|\xi|^2}P_d.
$$
Thus, from Lemma \ref{3dmvpbsp9}, we have the following eigenvalues and eigenfunctions for \eqref{1msp21}:
\bma
\bigg(L-\mathrm{i}\epsilon(v\cdot\xi)-\mathrm{i}\epsilon&\frac{v\cdot\xi}{1+|\xi|^2}P_d\bigg)\psi_j(\xi,\epsilon)=\lambda_j(|\xi|,\epsilon)\psi_j(\xi,\epsilon),\nonumber\\
\lambda_j(|\xi|,\epsilon)=\beta_j(|\xi|,\epsilon),\quad &\psi_j(\xi,\epsilon)=\mathbb{O}e_j(\xi,\epsilon),\quad j=-1,0,1,2,3.\nonumber
\ema
This proves the lemma.
\end{proof}


By virtue of Lemma \ref{3dmvpbsp6} and Lemma \ref{3dmvpbsp9}, we can analyze on the semigroup $S(t,\xi,\epsilon)=e^{\frac{tB_{\epsilon}(\xi)}{\epsilon^2}}$ precisely by using an argument similar to that of \cite{Li-4}.
\begin{lem}\label{3dmvpbsp12}
The semigroup $S(t,\xi,\epsilon)=e^{\frac{tB_{\epsilon}(\xi)}{\epsilon^2}}$ with $\xi\in\mathbb{R}^3$ has the following decomposition:
\bq\label{se1}
S(t,\xi,\epsilon)f=S_1(t,\xi,\epsilon)f+S_2(t,\xi,\epsilon)f,\quad f\in L^2_{\xi}(\mathbb{R}^3), \ t>0,
\eq
where
\bq\label{se2}
S_1(t,\xi,\epsilon)f=\sum^{3}_{j=-1}e^{\frac{t\lambda_j(|\xi|,\epsilon)}{\epsilon^2}}\big(f,\overline{\psi_j(\xi,\epsilon)}\big)_{\xi}\psi_j(\xi,\epsilon),\quad \epsilon|\xi|\leq r_0,
\eq
with $\big(\lambda_j(|\xi|,\epsilon),\psi_j(\xi,\epsilon)\big)$ being the eigenvalue and eigenfunction of the operator $B_{\epsilon}(\xi)$ given in Lemma \ref{3dmvpbsp10} for $\epsilon|\xi|\leq r_0$, and $S_2(t,\xi,\epsilon)f=S(t,\xi,\epsilon)f-S_1(t,\xi,\epsilon)f$ satisfies for two constants $\sigma_0>0$ and $C>0$ independent of $\xi$ and $\epsilon$ that
\bq\label{se3}
\|S_2(t,\xi,\epsilon)f\|_{\xi}\leq Ce^{-\frac{\sigma_0 t}{\epsilon^2}}\|f\|_{\xi},\quad t>0.
\eq
\end{lem}


\section{Fluid approximation of semigroup}\setcounter{equation}{0}
\label{sect3}
In this section, we give the first and second order fluid approximations of the semigroup $e^{\frac{tB_{\epsilon}}{\epsilon^2}}$, which will be used to prove the convergence and establish the convergence rate of the solution to the mVPB system \eqref{Pm1}-\eqref{Pm3} towards the solution to the NSPF system \eqref{NS1}-\eqref{NS5}.

Firstly,  introduce a function sapce  $ H^k_P\ (L^2_P=H^0_P)$ with the norm
\bmas\|f\|_{H^k_P}&=\(\intr (1+|\xi|^2)^k
\|\hat{f}\|^2_{\xi} d\xi \)^{1/2}\\
&=\(\intr (1+|\xi|^2)^k
     \(\|\hat{f}\|^2+\frac1{1+|\xi|^2}\lt|(\hat{f},\sqrt{M})\rt|^2\)d\xi
    \)^{1/2}.
\emas
For any $f_0\in L^2(\R^3_x\times \R^3_v)$, set
\bq
e^{\frac{tB_{\epsilon}}{\epsilon^2}}f_0=\(\mathcal{F}^{-1}e^{\frac{tB_{\epsilon}(\xi)}{\epsilon^2}}\mathcal{F}\)f_0.
\eq
By Lemma \ref{3dmvpbsp1}, it hold that
$$
\|e^{\frac{tB_{\epsilon}}{\epsilon^2}}f_0\|^2_{H^k_P}=\int_{\mathbb{R}^3}(1+|\xi|^2)^k\|e^{\frac{tB_{\epsilon}(\xi)}{\epsilon^2}}\hat{f}_0\|^2_{\xi}d\xi\leq \int_{\mathbb{R}^3}(1+|\xi|^2)^k\|\hat{f}_0\|^2_{\xi}d\xi=\|f_0\|^2_{H^k_P}.
$$
This means that the operator $\frac{B_{\epsilon}}{\epsilon^2}$ generates a strongly continuous contraction semigroup $e^{\frac{tB_{\epsilon}}{\epsilon^2}}$ in $H^k_P$, and therefore, $f(t,x,v)=e^{\frac{tB_{\epsilon}}{\epsilon^2}}f_0$ is a global solution to the linearized mVPB system \eqref{Bm1} for any $f_0\in H^k_P$.

\subsection{Semigroup of the linear NSPF system}
In this subsection, we are going to study the solution to the linear  NSPF system. Consider the following linear NSPF system of $(n,m,q)(t,x)$:
\bma
&\nabla_x\cdot m=0,\quad n+\sqrt{\frac{2}{3}}q+(I-\Delta_x)^{-1}n=0,\label{lns1}\\
&\partial_tm-\kappa_0\Delta_{x}m+\nabla_xp=H_1,\label{lns3}\\
&\partial_t\bigg(q-\sqrt{\frac{2}{3}}n\bigg)-\kappa_1\Delta_xq=H_2,\label{lns4}
\ema
where $H_1=\big(H^1_1,H^2_1,H^3_1\big)$ and $H_2$ are given functions, $p$ is the pressure satisfying $p=\Delta_x^{-1}\div_xH_1$, and the initial data $(n,m,q)(0)$ satisfies \eqref{INS1}. For any $U_0=U_0(x,v)\in N_0$, we define
\bq\label{lns4j1}
V(t,\xi)\hat{U}_0=\sum_{j=0,2,3}e^{-d_j(|\xi|)t}\big(\hat{U}_0,E_j(\xi)\big)_{\xi}E_j(\xi),
\eq
where $d_j(|\xi|)$ and $E_j(\xi)$, $j=0,2,3$ are defined by \eqref{sp4} and \eqref{egf1-1} respectively. Set
\bq\label{lns4j1-1}
V(t)U_0=\(\mathcal{F}^{-1}V(t,\xi)\mathcal{F}\)U_0.
\eq


Then, we can represent the solution to the NSPF system \eqref{lns1}-\eqref{lns4} by the semigroup $V(t)$ as follows.
\begin{lem}\label{3dmvpbfas1}
For any $f_0\in L^2$ and $H_i\in L^1_t(L^2_x)$, $i=1,2$, define
$$
U(t,x,v)=V(t)P_0f_0+\int^t_0V(t-s)H(s)ds,
$$
where
$$
H(t,x,v)=H_1(t,x)\cdot v\chi_0+H_2(t,x)\chi_4.
$$
Let $(n,m,q)=((U,\chi_0),(U,v\chi_0),(U,\chi_4))$. Then $(n,m,q)(t,x)\in L^{\infty}_t(L^2_x)$ is an unique global solution to the linear NSPF system \eqref{lns1}-\eqref{lns4} with the initial data $(n,m,q)(0)$ satisfies \eqref{INS1}.
\end{lem}
\begin{proof}
By taking Fourier translation to \eqref{lns1}-\eqref{lns4}, we have
\bma
&\mathrm{i}\xi\cdot\hat{m}=0,\quad \hat{n}+\frac{1}{1+|\xi|^2}\hat{n}+\sqrt{\frac{2}{3}}\hat{q}=0,\label{lnsf1}\\
&\partial_t\hat{m}+\kappa_0|\xi|^2\hat{m} =\mathbb{O}_1\hat{H}_1,\label{lnsf3}\\
&\partial_t\bigg(\hat{q}-\sqrt{\frac{2}{3}}\hat{n}\bigg)+\kappa_1|\xi|^2\hat{q}=\hat{H}_2,\label{lnsf4}
\ema
where the initial data $(\hat{n},\hat{m},\hat{q})(0)$ satisfies
\bq\label{lnsi1}
\hat{m}(0)=\mathbb{O}_1\big(P_0\hat{f}_0,v\chi_0\big),\quad \hat{q}(0)-\sqrt{\frac{2}{3}}\hat{n}(0)=\bigg(P_0\hat{f}_0,\chi_4-\sqrt{\frac{2}{3}}\chi_0\bigg)
\eq
with $\mathbb{O}_1=\mathbb{O}_1(\xi)$ being a projection defined by
\bq
\mathbb{O}_1y=y-\(y\cdot\frac{\xi}{|\xi|}\)\frac{\xi}{|\xi|}, \quad \forall y\in \mathbb{R}^3.
\eq

By \eqref{lnsf1} and \eqref{lnsf4}, we obtain
\bma
&\hat{n}=-\sqrt{\frac{2}{3}}\frac{1+|\xi|^2}{2+|\xi|^2}\hat{q},\label{lnsj1}\\ &\frac{5|\xi|^2+8}{3|\xi|^2+6}\partial_t\hat{q}+\kappa_1|\xi|^2\hat{q}=\hat{H}_2.\label{lnsj2}
\ema
It follows from \eqref{lnsj2}, \eqref{lnsi1}, \eqref{egf1-1} and \eqref{sp4} that
\bq\label{lnsj3}
\begin{split}
\hat{q}(t,\xi)&=e^{-d_0(|\xi|)t}\hat{q}(0)+\int^t_0e^{-d_0(|\xi|)(t-s)}\frac{3|\xi|^2+6}{5|\xi|^2+8}\hat{H}_2(s)ds\\
&=e^{-d_0(|\xi|)t}\big(P_0\hat{f}_0,E_{0}(\xi)\big)_{\xi}\big(E_0(\xi),\chi_4\big)\\
&\quad+\int^t_0e^{-d_0(|\xi|)(t-s)}\big(\hat{H}(s),E_0(\xi)\big)_{\xi}\big(E_0(\xi),\chi_4\big)ds.
\end{split}
\eq
This and \eqref{lnsj1} imply that
\bq\label{lnsj4}
\begin{split}
\hat{n}(t,\xi)&=e^{-d_0(|\xi|)t}\big(P_0\hat{f}_0,E_{0}(\xi)\big)_{\xi}\big(E_0(\xi),\chi_0\big)\\
&\quad+\int^t_0e^{-d_0(|\xi|)(t-s)}\big(\hat{H}(s),E_0(\xi)\big)_{\xi}\big(E_0(\xi),\chi_0\big)ds.
\end{split}
\eq
By \eqref{lnsf1} and \eqref{lnsf3}, we have
\bq\label{lnsj5}
\begin{split}
\hat{m}(t,\xi)&=e^{-d_2(|\xi|)t} \hat{m}(0)+\int^t_0e^{-d_2(|\xi|)(t-s)}\mathbb{O}_1\hat{H}_1(s)ds\\
&=\sum_{j=2,3}e^{-d_j(|\xi|)t}\big(P_0\hat{f}_0,E_{j}(\xi)\big)_{\xi}\big(E_j(\xi),v\chi_0\big)\\
&\quad+\sum_{j=2,3}\int^t_0e^{-d_j(|\xi|)(t-s)}\big(\hat{H}(s),E_j(\xi)\big)_{\xi}\big(E_j(\xi),v\chi_0\big)ds.
\end{split}
\eq
Noting that $(E_0(\xi),v\chi_0)=0$ and $(E_j(\xi),\chi_0)=(E_j(\xi),\chi_4)=0$, $j=2,3,$ we can prove the lemma by using \eqref{lnsj3}-\eqref{lnsj5}.
\end{proof}

We have the time decay rates of the semigroup  $V(t)$ as follows.
\begin{lem}\label{3dmvpbfas3}
For any $\alpha\in\mathbb{N}^3$ and any $u_0\in N_0$, we have
\bma
\|\partial^{\alpha}_xV(t)u_0\|_{L^2}&\leq C(1+t)^{-\frac{3+2m}{4}}(\|\partial^{\alpha}_xu_0\|_{L^2}+\|\partial^{\alpha'}_xu_0\|_{L^{2,1}}),\\
\|\partial^{\alpha}_xV(t)u_0\|_{L^{\infty}}&\leq C(1+t)^{-\frac{3}{4}}\beta_m(t)(\|\partial^{\alpha'}_xu_0\|_{L^{\infty}}+\|\partial^{\alpha'}_xu_0\|_{L^2}),
\ema
where $\alpha'\leq\alpha$, $m=|\alpha-\alpha'|$, $C>0$ is a constant and $\beta_0(t)=\ln\(2+\frac{1}{t}\)$, $\beta_m(t)=t^{-\frac{m}{2}}$, $m\geq1$. 
\end{lem}
\begin{proof}
By \eqref{lns4j1}, we have
\bma
&V(t,\xi)\hat{u}_0=e^{-d_0(|\xi|)t}R_0(\xi)\bigg(\sqrt{\frac{3}{2}}\hat{U}_0-\hat{U}_4\bigg)+e^{-d_2(|\xi|)t}\sum_{j=1}^3R_j(\xi)\hat{U}_j,\label{lnsp10}\\
&\frac{\xi}{1+|\xi|^2} (V(\xi,t)\hat{u}_0,\chi_0 )=e^{-d_0(|\xi|)t}R_4(\xi)\bigg(\sqrt{\frac{3}{2}}\hat{U}_0-\hat{U}_4\bigg),\label{lnsp11}
\ema
where $U_j=(u_0,\chi_j),j=0,1,2,3,4$ and
\bma
&R_{0}(\xi)=\frac{\sqrt{6}\big(1+|\xi|^2\big)}{5|\xi|^2+8}\chi_0-\frac{3|\xi|^2+6}{5|\xi|^2+8}\chi_4,\nonumber\\
&R_j(\xi)=v_j\chi_0-\frac{(v\cdot\xi)}{|\xi|^2}\xi_j\chi_0,\quad j=1,2,3,\quad R_4(\xi)=\frac{\sqrt{6}\xi}{5|\xi|^2+8}.\nonumber
\ema
Then, we can estimate $\partial^{\alpha}_xV(t)u_0$ precisely by using an same argument as that of Lemma 3.5 in \cite{Li-1}. Hence, we omit the detail of the proof for brevity.
\end{proof}

\begin{remark} The semigroup $V(t)$ of INSP system associated with mVPB sytem is slightly different to the one of the classical VPB system. Precisely, the density of $V(t)u_0$ behaves like
\bma
 \(V(t,\xi)\hat{u}_0,\chi_0 \) =O(1)  e^{- |\xi|^2t} ,
\ema
while the density of  $V_c(t)u_0$ associated with classical VPB system behaves like (cf. \cite{Li-1})
\bma
 \(V_c(t,\xi)\hat{u}_0,\chi_0 \) =O(1)  e^{- |\xi|^2t}|\xi|^2.
\ema
This implies that the decay rate of density $\|n(t)\|_{L^2_x}$ in Lemma \ref{3dmvpbfas3} is slower than that of $\|n(t)\|_{L^2_x}$ in \cite{Li-1}.
\end{remark}


\subsection{Fluid approximation of $e^{\frac{t B_{\eps}}{\eps^2}}$}
We have the time decay rates of the semigroup $e^{\frac{tB_{\epsilon}}{\epsilon^2}}$  as follows.
\begin{lem}\label{3dmvpbfas2}
For any $\epsilon\in(0,1)$, $\alpha\in\mathbb{R}^3$ and any $f_0\in L^2$, we have
\bma
&\|P_0\partial^{\alpha}_xe^{\frac{tB_{\epsilon}}{\epsilon^2}}f_0\|_{L^2}\leq C\((1+t)^{-\frac{3+2m}{4}}+e^{-\frac{\sigma_0 t}{\epsilon^2}}\)(\|\partial^{\alpha}_xf_0\|_{L^2}+\|\partial^{\alpha'}_xf_0\|_{L^{2,1}}),\label{ns321}\\
&\|P_1\partial^{\alpha}_xe^{\frac{tB_{\epsilon}}{\epsilon^2}}f_0\|_{L^2}\leq
C\(\epsilon(1+t)^{-\frac{5+2m}{4}}+e^{-\frac{\sigma_0 t}{\epsilon^2}}\)(\|\partial^{\alpha}_xf_0\|_{H^1}+\|\partial^{\alpha'}_xf_0\|_{L^{2,1}}),\label{ns322}
\ema
where $\alpha'\leq\alpha$, $m=|\alpha-\alpha'|$, $\sigma_0>0$ and $C>0$ are two constants independent of $\epsilon$.

Moreover, if $P_0f_0=0$, then
\bma
&\|P_0\partial^{\alpha}_xe^{\frac{tB_{\epsilon}}{\epsilon^2}}f_0\|_{L^2}\leq C\(\epsilon(1+t)^{-\frac{5+2m}{4}}+e^{-\frac{\sigma_0 t}{\epsilon^2}}\)(\|\partial^{\alpha}_xf_0\|_{H^1}+\|\partial^{\alpha'}_xf_0\|_{L^{2,1}}),\label{ns325}\\
&\|P_1\partial^{\alpha}_xe^{\frac{tB_{\epsilon}}{\epsilon^2}}f_0\|_{L^2}\leq
C\(\epsilon^2(1+t)^{-\frac{7+2m}{4}}+e^{-\frac{\sigma_0 t}{\epsilon^2}}\)(\|\partial^{\alpha}_xf_0\|_{H^2}+\|\partial^{\alpha'}_xf_0\|_{L^{2,1}}).\label{ns326}
\ema
\end{lem}
\begin{proof}
By Lemma \ref{3dmvpbsp12}, we have that for $j=0,1$,
\bma
\|P_j\partial^{\alpha}_xe^{\frac{tB_{\epsilon}}{\epsilon^2}}f_0\|^2_{L^2}&=\int_{\mathbb{R}^3}\|P_j\xi^{\alpha}e^{\frac{tB_{\epsilon}(\xi)}{\epsilon^2}}\hat{f}_0\|^2d\xi\nonumber\\
&\leq\int_{\{|\xi|\leq\frac{r_0}{\epsilon}\}}\|\xi^{\alpha}P_jS_1(t,\xi,\epsilon)\hat{f}_0\|^2d\xi+\int_{\mathbb{R}^3}\|\xi^{\alpha}S_2(t,\xi,\epsilon)\hat{f}_0\|^2d\xi.\label{lnsp1}
\ema

We have by \eqref{fs1} and Lemma \ref{3dmvpbsp12} that
\bq\label{lnsp3}
\int_{\mathbb{R}^3}\|\xi^{\alpha}S_2(t,\xi,\epsilon)\hat{f}_0\|^2d\xi\leq Ce^{-\frac{2\sigma_0 t}{\epsilon^2}}\int_{\mathbb{R}^3}(\xi^{\alpha})^2\|\hat{f}_0\|^2d\xi\leq Ce^{-\frac{2\sigma_0 t}{\epsilon^2}}\|\partial^{\alpha}_xf_0\|^2_{L^2}.
\eq
By Lemmas \ref{3dmvpbsp10}-\ref{3dmvpbsp12}, we have for $\epsilon|\xi|\leq r_0$,
\be\label{s111}
S_1(t,\xi,\epsilon)\hat{f}_0=\sum^{3}_{j=-1}e^{\frac{-\mathrm{i}|\xi|u_j t}{\epsilon}-d_j t+O(\epsilon|\xi|^3)t}\big(\hat{f}_0,  \overline{\psi_j(\xi,\eps)} \big)_{\xi} \psi_j(\xi,\eps).
\ee
Thus,
\bma
&\quad\int_{\{|\xi|\leq\frac{r_0}{\epsilon}\}}\|\xi^{\alpha}P_0S_1(t,\xi,\epsilon)\hat{f}_0\|^2d\xi\nnm\\
&\leq C\int_{\{|\xi|\leq\frac{r_0}{\epsilon}\}}e^{-2c|\xi|^2t}(\xi^{\alpha})^2\|\hat{f}_0\|^2d\xi\nnm\\
&\leq C\sup_{\xi\in\R^3}\|\xi^{\alpha'}\hat{f}_0\|^2\int_{\{|\xi|\leq1\}}e^{-2c|\xi|^2t}|\xi|^{2|\alpha-\alpha'|}d\xi+Ce^{-2ct}\int_{\{|\xi|\geq1\}}(\xi^{\alpha})^2\|\hat{f}_0\|^2d\xi\nnm\\
&\leq C(1+t)^{-\frac{3+2m}{2}}\big(\|\partial^{\alpha}_xf_0\|^2_{L^2}+\|\partial^{\alpha'}_xf_0\|^2_{L^{2,1}}\big),\label{lnsp4}
\ema
where $m=|\alpha-\alpha'|$ and $c>0$ is a constant.

By \eqref{s111}, we can obtain
\bma
\int_{\{|\xi|\leq\frac{r_0}{\epsilon}\}}\|\xi^{\alpha}P_1S_1(t,\xi,\epsilon)\hat{f}_0\|^2d\xi&\leq C\epsilon^{2}\int_{\{|\xi|\leq\frac{r_0}{\epsilon}\}}e^{-2c|\xi|^2t}|\xi|^2(\xi^{\alpha})^2\|\hat{f}_0\|^2d\xi\nnm\\
&\leq C\epsilon^2(1+t)^{-\frac{5+2m}{2}}\big(\|\partial^{\alpha}_xf_0\|^2_{H^1}+\|\partial^{\alpha'}_xf_0\|^2_{L^{2,1}}\big),\label{lnsp5}
\ema
where we had used the fact that $P_1\psi_j(\xi,\eps)=O(\eps|\xi|)$. Combining \eqref{lnsp1}-\eqref{lnsp5}, we prove \eqref{ns321} and \eqref{ns322}.

If $P_0f_0=0$, then it holds that for $\epsilon|\xi|\leq r_0$,
\be\label{s222}
S_1(t,\xi,\epsilon)\hat{f}_0=\mathrm{i} \epsilon\sum^{3}_{j=-1}e^{\frac{-\mathrm{i}|\xi|u_j t}{\epsilon}-d_j t+O(\epsilon|\xi|^3)t} \big(\hat{f}_0, P_1\overline{\psi_j(\xi,\eps)} \big) \psi_j(\xi,\eps) .
\ee
Thus, by Lemmas \ref{3dmvpbsp10}-\ref{3dmvpbsp12}, \eqref{s111} and \eqref{s222}, we have
\bma
\int_{\{|\xi|\leq\frac{r_0}{\epsilon}\}}\|\xi^{\alpha}P_0S_1(t,\xi,\epsilon)\hat{f}_0\|^2d\xi&\leq C\epsilon^{2}\int_{\{|\xi|\leq\frac{r_0}{\epsilon}\}}e^{-2c|\xi|^2t}|\xi|^2(\xi^{\alpha})^2\|\hat{f}_0\|^2d\xi\nnm\\
&\leq C\epsilon^{2}(1+t)^{-\frac{5+2m}{2}}\big(\|\partial^{\alpha}_xf_0\|^2_{H^1}+\|\partial^{\alpha'}_xf_0\|^2_{L^{2,1}}\big),\label{lnsp8}\\
\int_{\{|\xi|\leq\frac{r_0}{\epsilon}\}}\|\xi^{\alpha}P_1S_1(t,\xi,\epsilon)\hat{f}_0\|^2d\xi&\leq C\int_{\{|\xi|\leq\frac{r_0}{\epsilon}\}}\epsilon^{4}e^{-2c|\xi|^2t}|\xi|^4(\xi^{\alpha})^2\|\hat{f}_0\|^2d\xi\nnm\\
&\leq C\epsilon^4(1+t)^{-\frac{7+2m}{2}}\big(\|\partial^{\alpha}_xf_0\|^2_{H^2}+\|\partial^{\alpha'}_xf_0\|^2_{L^{2,1}}\big),\label{lnsp9}
\ema
where we had used the fact that $P_1\psi_j(\xi,\eps)=O(\eps|\xi|)$. Combining \eqref{lnsp1} and \eqref{lnsp8}-\eqref{lnsp9}, we obtain \eqref{ns325}-\eqref{ns326}.
\end{proof}


We now prepare three lemmas \ref{3dmvpbfas4}-\ref{3dmvpbfas5} that will be used to study the fluid dynamical approximation of the semigroup $e^{\frac{tB_{\epsilon}}{\epsilon^2}}$.
\begin{lem}\label{3dmvpbfas4}
For any functions $\phi(r)$ and $\varrho(r)$ satisfying $\big|\phi^{(k)}(r)\big|\leq C(1+r)^{-2-k-\delta}$ and $| (\frac{1}{\varrho'(r)})^{(k)}|\leq C(1+r)^{-k}$  $(k=0,1)$ for any $\delta>0$, we have
$$
\bigg|\int_{\mathbb{R}^3}e^{\mathrm{i}x\cdot\xi}e^{\mathrm{i}\vartheta\varrho(|\xi|)}\alpha(\omega)\phi(|\xi|)d\xi\bigg|\leq C|\vartheta|^{-1},
$$
where $\alpha(\omega)$ is a smooth function for $\omega=\frac{\xi}{|\xi|}\in\S^2$ and $\vartheta\in\R$.
\end{lem}
\begin{proof}
Firstly, note that
\bq\label{lmfs-4-1}
\int_{\mathbb{R}^3}e^{\mathrm{i}x\cdot\xi}e^{\mathrm{i}\vartheta\varrho(|\xi|)}\alpha(\omega)\phi(|\xi|)d\xi=\int_{0}^{+\infty}e^{\mathrm{i}\vartheta\varrho(r)}g(x,r)\phi(r)r^2dr,
\eq
where
$$
g(x,r)=\int_{\mathbb{S}^2}e^{\mathrm{i}rx\cdot\omega}\alpha(\omega)d\omega.
$$
By Lemma 3.7 in \cite{Yang-4}, we have
$$
|g(x,r)|\leq C(1+|x|r)^{-1},\quad \big|\partial_rg(x,r)\big|\leq C|x|(1+|x|r)^{-1}.
$$
Thus,
\bmas
&\quad\int_{\mathbb{R}^3}e^{\mathrm{i}x\cdot\xi}e^{\mathrm{i}\vartheta\varrho(|\xi|)}\alpha(\omega)\phi(|\xi|)d\xi\nnm\\
&=\frac{1}{\mathrm{i}\vartheta}\int_0^{+\infty}g(x,r)\phi(r)\frac1{\varrho'(r) }r^2de^{\mathrm{i}\vartheta\varrho(r)}\nnm\\
&=\frac{1}{\mathrm{i}\vartheta}e^{\mathrm{i}\vartheta\varrho(r)}g(x,r)\phi(r)\frac{r^2}{\varrho'(r) }\bigg|^{+\infty}_0-\frac{1}{\mathrm{i}\vartheta}\int_0^{+\infty}e^{\mathrm{i}\vartheta\varrho(r)}\phi(r)\frac{r^2}{\varrho'(r) }\partial_rg(x,r)dr\nnm\\
&\quad-\frac{1}{\mathrm{i}\vartheta}\int_0^{+\infty}e^{\mathrm{i}\vartheta\varrho(r)}g(x,r)\bigg(\frac{r^2}{\varrho'(r) } \phi'(r)+\phi(r) \frac{\varrho''(r)r^2}{\varrho'(r)^2 }+2\frac{\phi(r)r}{\varrho'(r) }\bigg)dr,
\emas
which yields
$$
\bigg|\int_{\mathbb{R}^3}e^{\mathrm{i}x\cdot\xi}e^{\mathrm{i}\vartheta\varrho(|\xi|)}\alpha(\omega)\phi(|\xi|)d\xi\bigg|\leq C|\vartheta|^{-1}.
$$
And this completes the proof of the lemma.
\end{proof}

\begin{lem}\label{3dmvpbfas4-1}
There exists a constant $C>0$ such that
$$
\bigg\|\int_{\R^3}e^{\mathrm{i}x\cdot\xi}e^{-d_j(|\xi|)t}d\xi\bigg\|_{L^1_x}\le C,
$$
where $d_j(|\xi|)$, $j=-1,0,1,2,3$ is given by \eqref{sp4}.
\end{lem}
\begin{proof}
For any $\alpha\in\N^3$ and $j=-1,0,1,2,3$, by \eqref{sp4} we obtain
$$
\Big|\partial_{\xi}^{\alpha}e^{-d_j(|\xi|)t}\Big|\leq Ct^{\frac{|\alpha|}{2}}\big(1+|\xi|^2t\big)^{\frac{|\alpha|}{2}}e^{-c|\xi|^2t},
$$
where $c,C>0$ are two constants. This implies that
$$
\bigg|x^{\alpha}\int_{\R^3}e^{\mathrm{i}x\cdot\xi}e^{-d_j(|\xi|)t}d\xi\bigg|\leq C\int|\partial_{\xi}^{\alpha}e^{-d_j(|\xi|)t}|d\xi
\leq Ct^{-\frac{3}{2}+\frac{|\alpha|}{2}},$$
which gives
$$
\big|\mathcal{F}^{-1}\big(e^{-d_j(|\xi|)t}\big)\big|\leq Ct^{-\frac{3}{2}}\(1+\frac{|x|^2}{t}\)^{-n}, \quad \forall n\geq2.
$$
This proves the lemma.
\end{proof}

\begin{lem}\label{3dmvpbfas5}
For any $f_0\in N_0$, we have
\bq\label{f-sS-2}
\|S_2(t,\xi,\epsilon)f_0\|_{\xi}\leq C\big(\epsilon|\xi|1_{\{\epsilon|\xi|\leq r_0\}}+1_{\{\epsilon|\xi|\geq r_0\}}\big)e^{-\frac{\sigma_0 t}{\epsilon^2}}\|f_0\|_{\xi}.
\eq
\end{lem}
\begin{proof}
Define a projection $P_{\epsilon}(\xi)$ by
$$
P_{\epsilon}(\xi)f=\sum^3_{j=-1}\big(f,\overline{\psi_j(\xi,\epsilon)}\big)_{\xi}\psi_j(\xi,\epsilon),\quad \forall f\in L^2\big(\R^3_v\big),
$$
where $\psi_j(\xi,\epsilon)$, $j=-1,0,1,2,3$ are the eigenfunctions of $B_{\epsilon}(\xi)$ defined by \eqref{egf1} for $\epsilon|\xi|\leq r_0$.

By Lemma \ref{3dmvpbsp12}, we can assert that
\bq
S_1(t,\xi,\epsilon)=e^{\frac{tB_{\epsilon}(\xi)}{\epsilon^2}}1_{\{\epsilon|\xi|\leq r_0\}}P_{\epsilon}(\xi).
\eq
Indeed, it follows from semigroup theory  that for $\epsilon|\xi|\leq r_0$ and $\kappa>0$,
\bmas
e^{\frac{tB_{\epsilon}(\xi)}{\epsilon^2}}P_{\epsilon}(\xi)f&=\frac{1}{2\pi\mathrm{i}}\int_{\kappa-\mathrm{i}\infty}^{\kappa+\mathrm{i}\infty}e^{\frac{t\lambda}{\epsilon^2}}(\lambda-B_{\epsilon}(\xi))^{-1}P_{\epsilon}(\xi)fd\lambda\\
&=\frac{1}{2\pi\mathrm{i}}\sum^3_{j=-1}\int_{\kappa-\mathrm{i}\infty}^{\kappa+\mathrm{i}\infty}e^{\frac{t\lambda}{\epsilon^2}}(\lambda-\lambda_j(|\xi|,\epsilon))^{-1}d\lambda\big(f,\overline{\psi_j(\xi,\epsilon)}\big)_{\xi}\psi_j(\xi,\epsilon)\\
&=\sum^3_{j=-1}e^{\frac{t\lambda_j(|\xi|,\epsilon)}{\epsilon^2}}\big(f,\overline{\psi_j(\xi,\epsilon)}\big)_{\xi}\psi_j(\xi,\epsilon)=S_1(t,\xi,\epsilon)f.
\emas

By \eqref{se3}, we have
\bq\label{f-sS-0}
S_2(t,\xi,\epsilon)=S_{21}(t,\xi,\epsilon)+S_{22}(t,\xi,\epsilon),
\eq
where
\bmas
S_{21}(t,\xi,\epsilon)&=S(t,\xi,\epsilon)1_{\{\epsilon|\xi|\leq r_0\}}\(I-P_{\epsilon}(\xi)\),\\
S_{22}(t,\xi,\epsilon)&=S(t,\xi,\epsilon)1_{\{\epsilon|\xi|\geq r_0\}}.
\emas
It holds that
\bq\label{f-sS-j}
\|S_{2j}(t,\xi,\epsilon)g\|_{\xi}\leq Ce^{-\frac{\sigma_0t}{\epsilon^2}}\|g\|_{\xi},\quad \forall g\in L^2\big(\R^3_v\big),\quad j=1,2.
\eq

Since $E_j(\xi)$, $j=-1,0,1,2,3$ are the orthonormal basis of $N_0$, it follows that for any $f_0\in N_0$,
$$
f_0-P_{\epsilon}(\xi)f_0=\sum^3_{j=-1}\big(f_0,E_j(\xi)\big)_{\xi}E_j(\xi)-\sum^3_{j=-1}\big(f_0,\overline{\psi_j(\xi,\epsilon)}\big)_{\xi}\psi_j(\xi,\epsilon),
$$
which together with \eqref{egf1} gives rise to
\bq\label{f-sS-1}
\|S_{21}(t,\xi,\epsilon)f_0\|_{\xi}\leq C\epsilon|\xi|1_{\{\epsilon|\xi|\leq r_0\}}e^{-\frac{\sigma_0t}{\epsilon^2}}\|f_0\|_{\xi}, \quad \forall f_0\in N_0.
\eq
By combining \eqref{f-sS-0}-\eqref{f-sS-1}, we obtain \eqref{f-sS-2}.
\end{proof}

Now we are going to estimate the first and second order expansions of the semigroup $e^{\frac{tB_{\epsilon}}{\epsilon^2}}$.
\begin{lem}\label{3dmvpbfas6}
For any $\epsilon\in (0,1)$ and any $f_0\in L^2(\R^3_x\times \R^3_v),$ we have
\be
\big\|e^{\frac{tB_{\epsilon}}{\epsilon^2}}f_0-V(t)P_0f_0\big\|_{L^{\infty}} \leq C\bigg(\epsilon(1+t)^{-2}+\(1+\frac{t}{\epsilon}\)^{-1}\bigg)\(\|f_0\|_{H^{3}}+\|f_0\|_{W^{3,1}}\),\label{lsp2j1}
\ee
where $V(t)$ is given in \eqref{lns4j1}, and $C>0$ is a constant independent of $\epsilon$. Moreover, if $f_0$ satisfies \eqref{F01}, then
\bq\label{lsp2j2}
\big\|e^{\frac{tB_{\epsilon}}{\epsilon^2}}f_0-V(t)P_0f_0\big\|_{L^{\infty}}\leq C\epsilon(1+t)^{-2}\(\|f_0\|_{H^{3}}+\|f_0\|_{W^{3,1}}\).
\eq
\end{lem}
\begin{proof}
First, we prove \eqref{lsp2j1} as follows. By Lemma \ref{3dmvpbsp12}, we have
\bma\label{F8-1-1}
\big\|e^{\frac{tB_{\epsilon}}{\epsilon^2}}f_0-V(t)P_0f_0\big\|
&=\bigg\|\int_{\mathbb{R}^3}e^{\mathrm{i}x\cdot\xi}\(e^{\frac{tB_{\epsilon}(\xi)}{\epsilon^2}}\hat{f}_0-V(t,\xi)P_0\hat{f}_0\)d\xi\bigg\|\nonumber\\
&\leq \bigg\|\int_{\{|\xi|\leq\frac{r_0}{\epsilon}\}}e^{\mathrm{i}x\cdot\xi}(S_1(t,\xi,\epsilon)\hat{f}_0-V(t,\xi)P_0\hat{f}_0)d\xi\bigg\|\nonumber\\
&\quad+\int_{\{|\xi|\geq\frac{r_0}{\epsilon}\}}\|V(t,\xi)P_0\hat{f}_0\|d\xi+\int_{\mathbb{R}^3}\|S_2(t,\xi,\epsilon)\hat{f}_0\|d\xi\nonumber\\
&=:I_1+I_2+I_3.
\ema
We estimates $I_j$, $j=1,2,3$ as follows. By Lemma \ref{3dmvpbsp10}-\ref{3dmvpbsp12}, we have
\bq
S_1(t,\xi,\epsilon)\hat{f}_0=\sum^3_{j=-1}e^{\frac{-\mathrm{i}|\xi|u_j t}{\epsilon}-d_j t+O(\epsilon|\xi|^3)t}\(\big(P_0\hat{f}_0,E_j(\xi)\big)_{\xi}E _j(\xi)+O(\epsilon|\xi|)\),
\eq
which leads to
\bma\label{F6}
I_1\leq&\sum^3_{j=-1}\int_{\{|\xi|\leq\frac{r_0}{\epsilon}\}}\Big\|e^{\frac{-\mathrm{i}|\xi|u_jt}{\epsilon}-d_jt+O(\epsilon|\xi|^3)t}\(\big(P_0\hat{f}_0,E_j(\xi)\big)_{\xi}E _j(\xi)+O(\epsilon|\xi|)\)\nonumber\\
&-e^{\frac{-\mathrm{i}|\xi|u_jt}{\epsilon}-d_jt}\big(P_0\hat{f}_0,E_{j}(\xi)\big)_{\xi}E_{j}(\xi)\Big\|d\xi\nonumber\\
&+\sum_{j=-1,1}\bigg\|\int_{\{|\xi|\leq\frac{r_0}{\epsilon}\}}e^{\mathrm{i}x\cdot\xi}e^{\frac{-\mathrm{i}|\xi|u_jt}{\epsilon}-d_jt}\big(P_0\hat{f}_0,E_j(\xi)\big)_{\xi}E_j(\xi)d\xi\bigg\|\nonumber\\
=:&I_{11}+I_{12}.
\ema

For $I_{11}$, it follows from \eqref{sp4} and \eqref{egf1-1} that
\bma\label{F7}
I_{11}&\leq C\epsilon\int_{\{|\xi|\leq\frac{r_0}{\epsilon}\}}e^{-c|\xi|^2t}(|\xi|^3t\|\hat{f}_0\|_{\xi}+|\xi|\|\hat{f}_0\|_{\xi})d\xi\nnm\\
&\leq C\epsilon\sup_{|\xi|\leq1}\|\hat{f}_0\|\int_{\{|\xi|\leq1\}}e^{-c|\xi|^2t}(|\xi|+|\xi|^3t)d\xi\nnm\\
&\quad+C\epsilon\(\int_{\{|\xi|>1\}}e^{-c|\xi|^2t}\frac{(1+|\xi|^2t)^2}{(1+|\xi|^2)^2}d\xi\)^{\frac{1}{2}}\(\int_{\{|\xi|>1\}}(1+|\xi|^2)^2|\xi|^2\|\hat{f}_0\|^2d\xi\)^{\frac{1}{2}}\nnm\\
&\leq C\epsilon(1+t)^{-2}\(\|f_0\|_{H^3}+\|f_0\|_{L^{2,1}}\),
\ema
where $c>0$ is a constant, and we have used
$$
\int_{\R^3}\frac{(\xi^{\alpha})^2}{1+|\xi|^2}\big|\big(\hat{f}_0,\chi_0\big)\big|^2d\xi\leq \int_{\R^3}(\xi^{\alpha})^2\big|\big(\hat{f}_0,\chi_0\big)\big|^2d\xi\leq \|\partial^{\alpha}_xf_0\|^2_{L^2}.
$$

Then, we estimate $I_{12}$ as follows.
\bma
I_{12}&=\sum_{j=-1,1}\bigg\|\bigg(\int_{\R^3}-\int_{\{|\xi|>\frac{r_0}{\epsilon}\}}\bigg)e^{\mathrm{i}x\cdot\xi}e^{\frac{-\mathrm{i}|\xi|u_jt}{\epsilon}-d_jt}\(P_0\hat{f}_0,E_j(\xi)\)_{\xi}E_j(\xi)d\xi\bigg\|\nnm\\
&\leq\sum_{j=-1,1}\bigg\|\int_{\R^3}e^{\mathrm{i}x\cdot\xi}\hat{H}_j(t,\xi)d\xi\bigg\|+\sum_{j=-1,1}\bigg\|\int_{\{|\xi|>\frac{r_0}{\epsilon}\}}e^{\mathrm{i}x\cdot\xi}\hat{H}_j(t,\xi)d\xi\bigg\|\nnm\\
&=:I_{12}^1+I_{12}^2,\label{F7-1}
\ema
where
$$
\hat{H}_j(t,\xi)=e^{\frac{-\mathrm{i}|\xi|u_jt}{\epsilon}-d_jt}\(P_0\hat{f}_0,E_j(\xi)\)_{\xi}E_j(\xi),\quad j=\pm1.
$$
For $I_{12}^1$ and $I_{12}^2$, we have
\bma
I_{12}^1&\leq C\int_{\R^3}e^{-c|\xi|^2t}\|P_0\hat{f}_0\|_{\xi}d\xi\nnm\\
&\leq C\(\int_{\R^3}e^{-2c|\xi|^2t}\frac{1}{(1+|\xi|^2)^2}d\xi\)^{\frac{1}{2}}\(\int_{\R^3}(1+|\xi|^2)^2\|\hat{f}_0\|^2d\xi\)^{\frac{1}{2}}\nnm\\
&\leq C\|f_0\|_{H^{2}},\label{F7-2}\\
I_{12}^2&\leq Ce^{-\frac{cr_0^2t}{\epsilon^2}}\int_{\{|\xi|>\frac{r_0}{\epsilon}\}}\|P_0\hat{f}_0\|_{\xi}d\xi\nnm\\
&\leq Ce^{-\frac{cr_0^2t}{\epsilon^2}}\bigg(\int_{\R^3}\frac{1}{(1+|\xi|^2)^2}d\xi\bigg)^{\frac{1}{2}}\bigg(\int_{\R^3}(1+|\xi|^2)^2\|\hat{f}_0\|_{\xi}^2d\xi\bigg)^{\frac{1}{2}}\nnm\\
&\leq Ce^{-\frac{cr_0^2t}{\epsilon^2}}\|f_0\|_{H^2}.\label{F7-3}
\ema

By \eqref{egf1-1}, we have
\bma
&\quad\big(P_0\hat{f}_0,E_j(\xi)\big)_{\xi}E_{j}(\xi)\nnm\\
&=\bigg[\sqrt{\frac{3}{2}}\frac{|\xi|^2+2}{5|\xi|^2+8}\hat{n}_0+\frac{|\xi|^2+1}{5|\xi|^2+8}\hat{q}_0-\sqrt{\frac{|\xi|^2+1}{5|\xi|^2+8}}\frac{j\sqrt{2}\hat{m}_0\cdot\omega}{2}\bigg]\bigg(\sqrt{\frac{3}{2}}\chi_0+\chi_4\bigg)\nnm\\
&\quad-\frac{j\sqrt{2}}{2}\bigg[\sqrt{\frac{3}{2}}\frac{|\xi|^2+2}{\sqrt{(|\xi|^2+1)(5|\xi|^2+8)}}\hat{n}_0+\sqrt{\frac{|\xi|^2+1}{5|\xi|^2+8}}\hat{q}_0-\frac{j\sqrt{2}\hat{m}_0\cdot\omega}{2}\bigg]\omega\cdot v\chi_0,\label{F7-4}
\ema
where $j=\pm1$, and
$$
\big(\hat{n}_0,\hat{m}_0,\hat{q}_0\big)=\big((\hat{f}_0,\chi_0),(\hat{f}_0,v\chi_0),(\hat{f}_0,\chi_4)\big).
$$
Thus,
\bma
\hat{H}_j&= \bigg(\sqrt{\frac{3}{2}}\hat{G}_{j}^{11}\hat{\mathcal{C}}_j\hat{F}_0+\hat{G}_{j}^{12}\hat{\mathcal{C}}_j\hat{F}_2-\frac{j\sqrt{2}}{2}\hat{G}_{j}^{23}\hat{\mathcal{C}}_j\cdot\hat{F}_1\bigg)\bigg(\sqrt{\frac{3}{2}}\chi_0+\chi_4\bigg) \nnm\\
&\quad-\bigg(\frac{j\sqrt{3}}{2}\hat{G}_{j}^{24}\hat{\mathcal{C}}_j\hat{F}_0+\frac{j\sqrt{2}}{2}\hat{G}_{j}^{23}\hat{\mathcal{C}}_j\hat{F}_2-\frac{1}{2}\hat{G}_{j}^{35}\hat{\mathcal{C}}_j\hat{F}_1\bigg)\cdot v\chi_0,\label{F7-6}
\ema
where $j=-1,1$, and
\bq\label{F7-6-1}
\left\{\bal
\hat{G}_{j}^{kl}(t,\xi)=e^{\frac{ \mathrm{i}|\xi|u_jt}{\epsilon}}\alpha_k(\omega)\frac{\mathcal{B}_l(|\xi|)}{(1+|\xi|)^{3}}, \quad  k=1,2,3,\,\, l=1,2,3,4,5,\\
\hat{\mathcal{C}}_j(t,\xi)=e^{-d_jt},\\
(\hat{F}_0,\hat{F}_1,\hat{F}_2)=(1+|\xi|)^3(\hat{n}_0,\hat{m}_0,\hat{q}_0),\\
\alpha_1(\omega)=1,\quad \alpha_2(\omega)=\omega,\quad \alpha_3(\omega)=\omega\otimes\omega,\\
\mathcal{B}_1(\xi)=\frac{|\xi|^2+2}{5|\xi|^2+8},\quad \mathcal{B}_2(\xi)=\frac{|\xi|^2+1}{5|\xi|^2+8},\quad \mathcal{B}_3(\xi)=\sqrt{\frac{|\xi|^2+1}{5|\xi|^2+8}},\\
\mathcal{B}_4(\xi)=\frac{|\xi|^2+2}{\sqrt{(|\xi|^2+1)(5|\xi|^2+8)}},\quad \mathcal{B}_5(\xi)=1.
\ea\right.
\eq
Thus, by Lemmas \ref{3dmvpbfas4}-\ref{3dmvpbfas4-1}, we have
\bma
I_{12}^1
&\leq   C\(\|G_{j}^{11}\|_{L^{\infty}_x}+\|G_{j}^{12}\|_{L^{\infty}_x}+\|G_{j}^{23}\|_{L^{\infty}_x}\)\|\mathcal{C}_j\|_{L^1_x}\|(F_0,F_1,F_2)\|_{W^{3,1}_x}\nnm\\
&\quad +C\(\|G_{j}^{23}\|_{L^{\infty}_x}+\|G_{j}^{24}\|_{L^{\infty}_x}+\|G_{j}^{35}\|_{L^{\infty}_x}\)\|\mathcal{C}_j\|_{L^1_x}\|(F_0,F_1,F_2)\|_{W^{3,1}_x} \nnm\\
&\leq C\(\frac{t}{\epsilon}\)^{-1}\|f_0\|_{W^{3,1}}.\label{F7-7}
\ema
By combining \eqref{F7-2}-\eqref{F7-3} and \eqref{F7-7}, we have
\bq
I_{12}\leq C\(1+\frac{t}{\epsilon}\)^{-1}\(\|f_0\|_{H^3}+\|f_0\|_{W^{3,1}}\),\label{F8}
\eq
which together with \eqref{F6} and \eqref{F7} implies that
\bq\label{F9}
I_1\leq C\bigg(\epsilon(1+t)^{-2}+\bigg(1+\frac{t}{\epsilon}\bigg)^{-1}\bigg)\(\|f_0\|_{H^3}+\|f_0\|_{W^{3,1}}\).
\eq

By Lemma \ref{3dmvpbsp12} and \eqref{lnsp10}, we have
\bma
I_2&\le C\eps \int_{\{|\xi|\geq\frac{r_0}{\epsilon}\}}e^{-c|\xi|^2t}|\xi|\|P_0\hat{f}_0\|_{\xi}d\xi \nnm\\
&\leq C\eps e^{-\frac{cr_0^2t}{\epsilon^2}}\(\int_{\{|\xi|\geq\frac{r_0}{\epsilon}\}}\frac{1}{(1+|\xi|^2)^2}d\xi\)^{\frac{1}{2}}\(\int_{\{|\xi|\geq\frac{r_0}{\epsilon}\}}(1+|\xi|^2)^3 \| \hat{f}_0\|^2d\xi\)^{\frac{1}{2}}\nonumber\\
&\leq C\eps e^{-\frac{cr_0^2t}{\epsilon^2}}\|f_0\|_{H^3},\label{F10}\\
I_3&\leq Ce^{-\frac{\sigma_0 t}{\epsilon^2}}\(\int_{\R^3}\frac{1}{(1+|\xi|^2)^2}d\xi\)^{\frac{1}{2}}\(\int_{\R^3}(1+|\xi|^2)^2\|\hat{f}_0\|_{\xi}^2d\xi\)^{\frac{1}{2}}\nnm\\
&\leq Ce^{-\frac{\sigma_0 t}{\epsilon^2}}\|f_0\|_{H^2}.\label{F11}
\ema
Therefore, it follows from \eqref{F9}-\eqref{F11} that
\be\label{F12}
\big\|e^{\frac{tB_{\epsilon}}{\epsilon^2}}f_0-V(t)P_0f_0\big\|_{L^{\infty}}\leq C\bigg(\epsilon(1+t)^{-2}+\(1+\frac{t}{\epsilon}\)^{-1}\bigg)\(\|f_0\|_{H^3}+\|f_0\|_{W^{3,1}}\).
\ee
Thus, we prove \eqref{lsp2j1}.

Finally, we turn to show \eqref{lsp2j2}. If $f_0$ satisfies \eqref{F01}, then we have
$$
\big(P_0\hat{f}_0,E_j(\xi)\big)_{\xi}=0,\quad j=\pm1,
$$
which implies that $I_{12}=0$. The term $I_{11}$ satisfies \eqref{F7}. It follows from Lemma \ref{3dmvpbfas5}  that
\bma
I_3&\leq C\epsilon\int_{\{|\xi|\leq\frac{r_0}{\epsilon}\}}e^{-\frac{\sigma_0 t}{\epsilon^2}}|\xi|\|\hat{f}_0\|d\xi +C\epsilon\int_{\{|\xi|\geq\frac{r_0}{\epsilon}\}}e^{-\frac{\sigma_0 t}{\epsilon^2}}|\xi|\|\hat{f}_0\|d\xi\nnm\\
&\leq C\eps e^{-\frac{\sigma_0 t}{\epsilon^2}}\|f_0\|_{H^3}. \label{I3}
\ema
Thus, by combining \eqref{F10} and \eqref{I3},  we can obtain \eqref{lsp2j2}. The proof of the lemma is completed.
\end{proof}


\begin{remark}
From Lemma \ref{3dmvpbfas6}, we have
$$
\|e^{\frac{tB_{\epsilon}}{\epsilon^2}}P_0f_0-V(t)P_0f_0-u_{\epsilon}^{osc}(t)\|_{L^{\infty}}\leq C\epsilon(1+t)^{-2}\(\|f_0\|_{H^{3}}+\|f_0\|_{W^{3,1}}\),
$$
where $u_{\epsilon}^{osc}(t)=u_{\epsilon}^{osc}(t,x,v)$ is the high oscillation part of $e^{\frac{tB_{\epsilon}}{\epsilon^2}}f_0$ defined by \eqref{dlxg-uosc}.
\end{remark}


\begin{lem}\label{3dmvpbfas7}
For any $\epsilon\in(0,1)$ and any $f_0\in L^2$ satisfying $P_0f_0=0$, we have
\bma
&\quad\bigg\|\frac{1}{\epsilon}e^{\frac{tB_{\epsilon}}{\epsilon^2}}f_0-V(t)P_0(v\cdot\nabla_xL^{-1}f_0)\bigg\|_{L^{\infty}}\nnm\\
&\leq C\bigg(\epsilon(1+t)^{-\frac{5}{2}}+\(1+\frac{t}{\epsilon}\)^{-1}+\frac{1}{\epsilon}e^{-\frac{\sigma_0 t}{\epsilon^2}}\bigg)\(\|f_0\|_{H^{4}}+\|f_0\|_{W^{4,1}}\).
\ema
where $V(t)$ is given in \eqref{lns4j1},  and $C>0$ is a constant independent of $\epsilon$.
\end{lem}

\begin{proof}
By Lemma \ref{3dmvpbsp12}, we obtain
\bma
&\quad\bigg\|\frac{1}{\epsilon}e^{\frac{tB_{\epsilon}}{\epsilon^2}}f_0-V(t)P_0(v\cdot\nabla_xL^{-1}f_0)\bigg\|\nonumber\\
&\leq \bigg\|\int_{\{|\xi|\leq\frac{r_0}{\epsilon}\}}e^{\mathrm{i}x\cdot\xi}\(\frac{1}{\epsilon}S_1(t,\xi,\epsilon)\hat{f}_0-V(t,\xi)P_0(\mathrm{i}v\cdot\xi L^{-1}\hat{f}_0)\)d\xi\bigg\|\nonumber\\
&\quad+\int_{\{|\xi|\geq\frac{r_0}{\epsilon}\}}\big\|V(t,\xi)P_0(\mathrm{i}v\cdot\xi L^{-1}\hat{f}_0)\big\|d\xi+\int_{\R^3}\bigg\|\frac{1}{\epsilon}S_2(t,\xi,\epsilon)\hat{f}_0\bigg\|d\xi\nonumber\\
&=:I_1+I_2+I_3.\label{fas7-0}
\ema
We estimate $I_j$, $j=1,2,3$ as follows. By Lemmas \ref{3dmvpbsp10}-\ref{3dmvpbsp12}, for any $f_0\in L^{2}$ satisfying $P_0f_0=0$, we have
\bq
S_1(t,\xi,\epsilon)\hat{f}_0=\mathrm{i}\epsilon\sum^3_{j=-1}e^{\frac{-\mathrm{i}|\xi|u_jt}{\epsilon}-d_jt+O(\epsilon|\xi|^3)t}\(\big(v\cdot\xi L^{-1}\hat{f}_0,E_j(\xi)\big)E_j(\xi)+O(\epsilon|\xi|^2)\),
\eq
which leads to
\bma
I_1&\leq \sum^3_{j=-1}\int_{\{|\xi|\leq \frac{r_0}{\epsilon}\}}\Big\|e^{\frac{-\mathrm{i}|\xi|u_jt}{\epsilon}-d_jt+O(\epsilon|\xi|^3)t}\(\big(v\cdot\xi L^{-1}\hat{f}_0,E_j(\xi)\big)E_j(\xi)+O(\epsilon|\xi|^2)\)\nnm\\
&\quad-e^{\frac{-\mathrm{i}|\xi|u_jt}{\epsilon}-d_jt}\big(v\cdot\xi L^{-1}\hat{f}_0,E_{j}(\xi)\big)E_{j}(\xi)\Big\|d\xi\nnm\\
&\quad+\sum_{j=-1,1}\bigg\|\int_{\{|\xi|\leq\frac{r_0}{\epsilon}\}}e^{\mathrm{i}x\cdot\xi}e^{\frac{-\mathrm{i}|\xi|u_jt}{\epsilon}-d_jt}\big(v\cdot\xi L^{-1}\hat{f}_0,E_j(\xi)\big)E_j(\xi)d\xi\bigg\|\nnm\\
&=:I_{11}+I_{12}.\label{fas7-1}
\ema

For $I_{11}$, it hold that
\bma
I_{11}&\leq C\epsilon\int_{\{|\xi|\leq\frac{r_0}{\epsilon}\}}e^{-c|\xi|^2t}\(|\xi|^3t\|P_0(v\cdot\xi L^{-1}\hat{f}_0)\|+|\xi|^2\|\hat{f}_0\|\)d\xi\nnm\\
&\leq C\epsilon(1+t)^{-\frac{5}{2}}\(\|f_0\|_{H^4}+\|f_0\|_{L^{2,1}}\).\label{fas7-1-1}
\ema
Then, we estimate $I_{12}$ as follows.
\bma
I_{12}&\leq\sum_{j=-1,1}\bigg\|\int_{\R^3}e^{\mathrm{i}x\cdot\xi}\hat{J}_j(t,\xi)d\xi\bigg\|+\sum_{j=-1,1}\bigg\|\int_{\{|\xi|>\frac{r_0}{\epsilon}\}}e^{\mathrm{i}x\cdot\xi}\hat{J}_j(t,\xi)d\xi\bigg\|\nnm\\
&=:I_{12}^1+I_{12}^2,\label{fas7-1-2}
\ema
where
$$
\hat{J}_j(t,\xi)=e^{\frac{-\mathrm{i}|\xi|u_jt}{\epsilon}-d_jt}\big(v\cdot\xi L^{-1}\hat{f}_0,E_j(\xi)\big)E_j(\xi),\quad j=\pm1.
$$
For $I_{12}^j$, $j=1,2$, we obtain by  a similar argument as \eqref{F7-2}-\eqref{F7-3} that
\bma
I_{12}^1&\leq C\int_{\R^3}e^{-c|\xi|^2t}|\xi|\|\hat{f}_0\|d\xi\leq C\|f_0\|_{H^{3}},\label{fas7-1-3}\\
I_{12}^2&\leq Ce^{-\frac{cr_0^2t}{\epsilon^2}}\int_{\{|\xi|>\frac{r_0}{\epsilon}\}}|\xi|\|\hat{f}_0\|d\xi\leq Ce^{-\frac{cr_0^2t}{\epsilon^2}}\|f_0\|_{H^3}.\label{fas7-1-4}
\ema
By \eqref{egf1-1} and $(v\cdot\xi L^{-1}\hat{f}_0,\chi_0)=0$, we have
\bma
&\quad\big(v\cdot\xi L^{-1}\hat{f}_0,E_j(\xi)\big)E_{j}(\xi)\nnm\\
&=\bigg[\frac{|\xi|^2+1}{5|\xi|^2+8}(v\cdot\xi L^{-1}\hat{f}_0,\chi_4)-\sqrt{\frac{|\xi|^2+1}{5|\xi|^2+8}}\frac{j\sqrt{2}(v\cdot\xi L^{-1}\hat{f}_0,v\chi_0)\cdot\omega}{2}\bigg]\bigg(\sqrt{\frac{3}{2}}\chi_0+\chi_4\bigg)\nnm\\
&\quad-\frac{j\sqrt{2}}{2}\bigg[\sqrt{\frac{|\xi|^2+1}{5|\xi|^2+8}}(v\cdot\xi L^{-1}\hat{f}_0,\chi_4)-\frac{j\sqrt{2}(v\cdot\xi L^{-1}\hat{f}_0,v\chi_0)\cdot\omega}{2}\bigg]\omega\cdot v\chi_0, \quad j=\pm 1. \label{fas7-1-5}
\ema
Thus,
\be
 \hat{J}_j = \bigg(\hat{G}_{j}^{12}\hat{\mathcal{C}}_j\hat{F}_4-\frac{j\sqrt{2}}{2}\hat{G}_{j}^{23}\hat{\mathcal{C}}_j\cdot\hat{F}_5\bigg)\bigg(\sqrt{\frac{3}{2}}\chi_0+\chi_4\bigg)
 -\bigg(\frac{j\sqrt{2}}{2}\hat{G}_{j}^{23}\hat{\mathcal{C}}_j\hat{F}_4-\frac{1}{2}\hat{G}_{j}^{35}\hat{\mathcal{C}}_j\hat{F}_5\bigg)\cdot v\chi_0,\label{fas7-1-7}
\ee
where $\hat{G}_{j}^{kl}$, $\hat{\mathcal{C}}_j$ ($j=\pm1$) are given in \eqref{F7-6-1}, and
$$
(\hat{F}_4,\hat{F}_5)=(1+|\xi|)^3\big((v\cdot\xi L^{-1}\hat{f}_0,\chi_4),(v\cdot\xi L^{-1}\hat{f}_0,v\chi_0)\big).
$$
Thus, by Lemmas \ref{3dmvpbfas4}-\ref{3dmvpbfas4-1}, we have
 \bma
I_{12}^1
&\leq C\(\|G_{j}^{12}\|_{L^{\infty}_x}+\|G_{j}^{23}\|_{L^{\infty}_x}+\|G_{j}^{35}\|_{L^{\infty}_x}\)\|\mathcal{C}_j\|_{L^1_x}\|(F_4,F_5)\|_{W^{3,1}_x}\nnm\\
&\leq C\(\frac{t}{\epsilon}\)^{-1}\|f_0\|_{W^{4,1}}.\label{fas7-1-8}
\ema
By combining \eqref{fas7-1-3}-\eqref{fas7-1-4} and \eqref{fas7-1-8}, we have
\bq
I_{12}\leq C\(1+\frac{t}{\epsilon}\)^{-1}\(\|f_0\|_{H^4}+\|f_0\|_{W^{4,1}}\),\label{fas7-2}
\eq
which together with \eqref{fas7-1}-\eqref{fas7-1-1} and \eqref{fas7-2} implies that
\bq\label{fas7-3}
I_1\leq C\bigg(\epsilon(1+t)^{-\frac{5}{2}}+\(1+\frac{t}{\epsilon}\)^{-1}\bigg)\(\|f_0\|_{H^4}+\|f_0\|_{W^{4,1}}\).
\eq

By Lemma \ref{3dmvpbsp12} and \eqref{lnsp10}, we have
\bma
I_2&\leq Ce^{-\frac{cr_0^2t}{\epsilon^2}}\(\int_{\{|\xi|\geq\frac{r_0}{\epsilon}\}}\frac{1}{(1+|\xi|^2)^2}d\xi\)^{\frac{1}{2}}\(\int_{\{|\xi|\geq\frac{r_0}{\epsilon}\}}(1+|\xi|^2)^2|\xi|^2\|\hat{f}_0\|^2d\xi\)^{\frac{1}{2}}\nonumber\\
&\leq Ce^{-\frac{cr_0^2t}{\epsilon^2}}\|f_0\|_{H^3},\label{fas7-4}\\
I_3&\leq Ce^{-\frac{\sigma_0 t}{\epsilon^2}}\frac{1}{\epsilon}\(\int_{\R^3}\frac{1}{(1+|\xi|^2)^2}d\xi\)^{\frac{1}{2}}\(\int_{\R^3}(1+|\xi|^2)^2\|\hat{f}_0\|^2d\xi\)^{\frac{1}{2}}\nnm\\
&\leq C\frac{1}{\epsilon}e^{-\frac{\sigma_0 t}{\epsilon^2}}\|f_0\|_{H^2}.\label{fas7-5}
\ema
Therefore, it follows from \eqref{fas7-0} and \eqref{fas7-3}-\eqref{fas7-5} that
\bmas
&\quad\bigg\|\frac{1}{\epsilon}e^{\frac{tB_{\epsilon}}{\epsilon^2}}f_0-V(t)P_0\big(v\cdot\nabla_xL^{-1}f_0\big)\bigg\|_{L^{\infty}}\nnm\\
&\leq C\bigg(\epsilon(1+t)^{-\frac{5}{2}}+\(1+\frac{t}{\epsilon}\)^{-1}+\frac{1}{\epsilon}e^{-\frac{\sigma_0 t}{\epsilon^2}}\bigg)\(\|f_0\|_{H^{4}}+\|f_0\|_{W^{4,1}}\).
\emas
The proof of the lemma is completed.
\end{proof}


\section{Diffusion Limit}\setcounter{equation}{0}
\label{sect4}

In this section, we study the diffusion limit of the solution to the nonlinear mVPB system \eqref{Pm1}-\eqref{Pm3} based on the fluid approximations of the semigroup given in Section \ref{sect3}. The solution $f_{\epsilon}(t)=f_{\epsilon}(t,x,v)$ to the mVPB system \eqref{Pm1}-\eqref{Pm3} can be represented by
\bma
f_{\epsilon}(t)=e^{\frac{tB_{\epsilon}}{\epsilon^2}}f_0+\int_0^te^{\frac{(t-s)B_{\epsilon}}{\epsilon^2}}\(G_1 +\frac{1}{\epsilon}G_2 +\frac{1}{\epsilon^2}v\sqrt{M}\cdot\nabla_x(I-\Delta_x)^{-1}G_3 \)(s)ds,\label{nldh1}
\ema
where the nonlinear terms $G_1(f_{\epsilon})$, $G_2(f_{\epsilon})$ and $G_3(f_{\epsilon})$ are defined in \eqref{3dmvpbnlt}.

Let $(n,m,q)(t,x)$ be the global solution to the NSPF system \eqref{NS1}-\eqref{NS5}. Then by Lemma \ref{3dmvpbfas1}, $u(t,x,v)=n(t,x)\chi_0+m(t,x)\cdot v\chi_0+q(t,x)\chi_4$ can be represented by
\bq
u(t)=V(t)P_0f_0+\int^t_0V(t-s)\big(Z_1(u)+\div_xZ_2(u)\big)(s)ds,\label{nlnsdh1}
\eq
where
\bma
&Z_1(u)=(n\nabla_x\phi)\cdot v\chi_0+\sqrt{\frac{2}{3}}(m\cdot\nabla_x\phi)\chi_4,\label{nlnsz1}\\
&Z_2(u)=-(m\otimes m)\cdot v\chi_0-\frac{5}{3}(qm)\chi_4.\label{nlnsz2}
\ema

\subsection{Energy estimate}

Let $N$ and $k$ be two positive integers. For hard sphere model and hard potential model, define two functionals $E_{N,k}(f_{\epsilon})$ and $D_{N,k}(f_{\epsilon})$ by
\bmas
&E_{N,k}(f_{\epsilon})=\sum_{|\alpha|+|\beta|\leq N}\big\|w_k\partial_x^{\alpha}\partial_v^{\beta}f_{\epsilon}\big\|^2_{L^2}+\sum_{|\alpha|\leq N}\big\|\partial_x^{\alpha}\phi_{\epsilon}\big\|^2_{H^1_x},\\
&D_{N,k}(f_{\epsilon})=\sum_{|\alpha|+|\beta|\leq N}\frac{1}{\epsilon^2}\big\|\nu^{\frac{1}{2}}w_k\partial_x^{\alpha}\partial_v^{\beta}P_1f_{\epsilon}\big\|^2_{L^2}+\sum_{|\alpha|\leq N-1}\(\big\|\partial_x^{\alpha}\nabla_xP_0f_{\epsilon}\big\|^2_{L^2}+\big\|\partial_x^{\alpha}\nabla_x\phi_{\epsilon}\big\|_{H^1_x}\),
\emas
where $w_k=w_k(v)$ is given by \eqref{hsw} for hard sphere model, and $w_k=w_k(t,v)$ is given by \eqref{hpw} for hard potential model. By the similar argument as \cite{Duan-3,Guo-2,Li-4}, we have the following lemma.
\begin{lem}\label{rbedl-ee2}
For $N\geq4$ and any $\epsilon\in(0,1)$, there exists a equivalent energy functional $\mathcal{E}_{N,1}(\cdot)\sim E_{N,1}(\cdot)$ with $w_1=w_1(v)$ is given by \eqref{hsw} for hard sphere model, and $w_1=w_1(t,v)$ is given by \eqref{hpw} for hard potential model, if $E_{N,1}(f_0)$ is sufficiently small, then the  mVPB system \eqref{Pm1}-\eqref{Pm3} admits a unique global solution $f_{\epsilon}=f_{\epsilon}(t,x,v)$  satisfying
\be
\Dt\mathcal{E}_{N,1}(f_{\epsilon}(t))+D_{N,1}(f_{\epsilon}(t))\leq 0.
\ee
\end{lem}

Then, by Lemma \ref{rbedl-ee2}, we have

\begin{lem}\label{rbedl5}
Let $N\geq4$, $a>0$, $0<b\leq\frac{1}{4}$. For any $\epsilon\in(0,1)$, there exists a small constant $\delta_0>0$ such that if $E_{N,1}(f_0)+\|f_0\|^2_{L^{2,1}}\leq\delta_0^2$, then the solution $f_{\epsilon}=f_{\epsilon}(t,x,v)$ to the mVPB system \eqref{Pm1}-\eqref{Pm3} has the following time-decay rate estimate:
\be
\|f_{\epsilon}\|_{H^N_{w_1}}+\|\phi_{\epsilon}\|_{ H^{N+1}_x}\leq C\delta_0(1+t)^{-\frac{3}{4}},\label{rbedl5-1}
\ee
where $C>0$ is a constant independent of $\epsilon$. In particular, we have
\bq\label{rbedl5-3}
\|P_1f_{\epsilon}(t)\|_{H^{N-3}}\leq C\delta_0\(\epsilon(1+t)^{-\frac{5}{4}}+e^{ -\frac{\sigma_0 t}{\epsilon^2}}\),
\eq
where $\sigma_0,C>0$ are two constants independent of $\epsilon$.
\end{lem}
\begin{proof}
First, we prove \eqref{rbedl5-1} as follows. Define
$$
Q_{\epsilon}(t)=\sup_{0\leq s\leq t}\Big\{(1+s)^{\frac{3}{4}}E_{N,1}(f_{\epsilon}(s))^{\frac{1}{2}}\Big\}.
$$
We claim that
\bq\label{rbedl5-0}
Q_{\epsilon}(t)\leq C\delta_0.
\eq
It is straightforward to verify that the estimate \eqref{rbedl5-1} follows from \eqref{rbedl5-0}.

From \cite{Duan-3,Guo-2}, we have
$$
\| \Gamma(f,g)\|\leq C \(\|w_1 f\|\| g\|+\| f\|\|w_1 g\|\) .
$$
Thus,
\be
\big\|\partial_x^{\alpha}G_j(f_{\epsilon}(s))\big\|_{L^2}+\big\|\partial_x^{\alpha}G_j(f_{\epsilon}(s))\big\|_{L^{2,1}}\leq CQ_{\epsilon}(t)^2(1+s)^{-\frac{3}{2}},\quad j=1,2,\label{rbedl5-1-2}
\ee
where $0\leq|\alpha|\leq N-1$, $0\leq s\leq t$ and $C>0$ is a constant.

Since it holds for $p\in[1,\infty] $ and $k=0,1$ that
\bma\label{youngbds}
\big\|\Tdx^{k} (I-\Delta_x)^{-1}f\big\|_{L^p_x}\leq \left\|\Tdx^{k}\((4\pi|x|)^{-1}e^{-|x|}\)\right\|_{L^1_x}\|f\|_{L^p_x} \le C\|f\|_{L^p_x},
\ema
we have
\be
\big\|\partial_x^{\alpha}\nabla_x(I-\Delta_x)^{-1}G_3(f_{\epsilon}(s))\big\|_{L^2_x}+ \big\|\partial_x^{\alpha}\nabla_x(I-\Delta_x)^{-1}G_3(f_{\epsilon}(s))\big\|_{L^{1}_x}\leq CQ_{\epsilon}(t)^2\epsilon^2(1+s)^{-\frac{3}{2}},\label{rbedl5-1-4}
\ee
where $0\leq|\alpha|\leq N$, $0\leq s\leq t$ and $C>0$ is a constant.

Since $P_0G_2(f_{\epsilon})=0$, it follows from Lemma \ref{3dmvpbfas2}, \eqref{nldh1} and \eqref{rbedl5-1-2}-\eqref{rbedl5-1-4} that
\bma
\|P_0f_{\epsilon}(t)\|_{L^2}&\leq C\((1+t)^{-\frac{3}{4}}+e^{-\frac{\sigma_0t}{\epsilon^2}}\)\(\|f_0\|_{L^2}+\|f_0\|_{L^{2,1}}\)\nnm\\
&\quad+C\int^{t}_0\((1+t-s)^{-\frac{3}{4}}+\frac{1}{\epsilon}e^{-\frac{\sigma_0(t-s)}{\epsilon^2}}\) \nnm\\
&\qquad \times\bigg(\sum^2_{j=1}\|G_j(s)\|_{H^1\cap L^{2,1}}+\frac{1}{\epsilon^2} \|\nabla_x(I-\Delta_x)^{-1}G_3(s) \|_{H^1_x\cap L^1_x} \bigg)ds\nnm\\
&\leq C\delta_0(1+t)^{-\frac{3}{4}}+CQ_{\epsilon}(t)^2(1+t)^{-\frac{3}{4}}.\label{rbedl5-2-1}
\ema

By Lemma \ref{rbedl-ee2}, \eqref{rbedl5-2-1} and using the fact that $D_{N,1}(f_{\epsilon})+\|P_0f_{\epsilon}\|^2_{L^2}\geq c_1\mathcal{E}_{N,1}(f_{\epsilon})$ for some constant $c_1>0$, we can obtain
\bma
E_{N,1}(f_{\epsilon}(t))&\leq Ce^{-c_1t}E_{N,1}(f_{0})+C\int_0^te^{-c_1(t-s)}\|P_0f_{\epsilon}(s)\|^2_{L^2}ds\nnm\\
&\leq C\delta_0^2e^{-c_1t}+C\(\delta_0+Q_{\epsilon}(t)^2\)^2\int_0^te^{-c_1(t-s)}(1+s)^{-\frac{3}{2}}ds\nnm\\
&\leq C(\delta_0+Q_{\epsilon}(t)^2)^2(1+t)^{-\frac{3}{2}}.\label{rbedl5-2-2}
\ema
By combining \eqref{rbedl5-2-1}-\eqref{rbedl5-2-2}, we obtain
$$
Q_{\epsilon}(t)\leq C\delta_0+CQ_{\epsilon}(t)^2,
$$
which leads to \eqref{rbedl5-0} provided $\delta_0>0$ sufficiently small. Then we proved \eqref{rbedl5-1}.

Finally, we deal with \eqref{rbedl5-3}. By Lemma \ref{3dmvpbfas2}, \eqref{nldh1} and \eqref{rbedl5-1-2}-\eqref{rbedl5-1-4} that
\bma
\| P_1f_{\epsilon}(t)\|_{H^k}&\leq C\(\epsilon(1+t)^{-\frac{5}{4}}+e^{-\frac{\sigma_0t}{\epsilon^2}}\)\(\| f_0\|_{H^{1+k}}+\|f_0\|_{L^{2,1}}\)\nnm\\
&\quad+C\int^{t}_0\(\epsilon(1+t-s)^{-\frac{5}{4}}+\frac{1}{\epsilon}e^{-\frac{\sigma_0(t-s)}{\epsilon^2}}\) \nnm\\
&\qquad \times\bigg(\sum^2_{j=1}\| G_j(s)\|_{H^{2+k}\cap L^{2,1}}+\frac{1}{\epsilon^2} \| \nabla_x(I-\Delta_x)^{-1}G_3(s) \|_{H^{2+k}_x\cap L^1_x} \bigg)ds\nnm\\
&\leq \(C\delta_0+CQ_{\epsilon}(t)^2\)\(\epsilon(1+t)^{-\frac{5}{4}}+e^{-\frac{\sigma_0t}{\epsilon^2}}\),\label{rbedl5-3-1}
\ema
where $k\leq N-3$. The proof of the lemma is completed.
\end{proof}

By virtue of Lemma \ref{3dmvpbfas3} and \eqref{nlnsdh1}, we have following lemma precisely by using an argument similar to that of \cite{Li-1}.

\begin{lem}\label{rbedl6}
Let $N\geq2$. There exists a small constant $\delta_0>0$ such that if $\|f_0\|_{H^N}+\|f_0\|_{L^{2,1}}\leq\delta_0$ then the NSPF system \eqref{NS1}-\eqref{INS1} admits a unique global solution $(n,m,q)(t,x)\in L^{\infty}_t\big(H^N_x\big)$.
Moreover, $u(t,x,v)=n(t,x)\chi_0+m(t,x)\cdot v\chi_0+q(t,x)\chi_4$ has the following time-decay rate:
$$
\|u(t)\|_{H^N}+\|\nabla_x\phi(t)\|_{H^N_x}\leq C\delta_0(1+t)^{-\frac{3}{4}},
$$
where $\phi(t,x)=-(I-\Delta_x)^{-1}n(t,x)$ and $C>0$ is a constant.
\end{lem}

\subsection{ Optimal convergence rate}
In this subsection, we will complete the proof of Theorem \ref{thm-2} about the convergence rate of the diffusion limit.
\begin{lem}[\cite{Li-1}]\label{ocr-1}
For any $i,j=1,2,3$, it hold that
\bmas
\Gamma(v_i\chi_0,v_j\chi_0)&=-\frac{1}{2}LP_1(v_iv_j\chi_0),\\
\Gamma(v_i\chi_0,|v|^2\chi_0)&=-\frac{1}{2}LP_1(v_i|v|^2\chi_0),\\
\Gamma(|v|^2\chi_0,|v|^2\chi_0)&=-\frac{1}{2}LP_1(|v|^4\chi_0).
\emas
\end{lem}

\noindent\textbf{\underline{Proof of Theorem 1.2.}} Firstly, we prove \eqref{thm-2-1}. Define
\bma\label{dlth2-0}
\Lambda_{\epsilon}(t)&=\sup_{0\leq s\leq t}\bigg\{\(\epsilon|\ln\epsilon|^2(1+s)^{-\frac{3}{4}}+\(1+\frac{s}{\epsilon}\)^{-1}\)^{-1}\nnm\\
&\qquad\qquad\times\(\|f_{\epsilon}(s)-u(s)\|_{L^{\infty}}+\|\nabla_x\phi_{\epsilon}(s)-\nabla_x\phi(s)\|_{L^{\infty}_x}\)\bigg\}.
\ema

By \eqref{nldh1} and \eqref{nlnsdh1}, we have
\bma\label{dlth2-1}
\|f_{\epsilon}(t)-u(t)\|_{L^{\infty}}&\leq \Big\|e^{\frac{tB_{\epsilon}}{\epsilon^2}}f_0-V(t)P_0f_0\Big\|_{L^{\infty}}+\int_0^t\Big\|e^{\frac{(t-s)B_{\epsilon}}{\epsilon^2}}G_1(f_{\epsilon})-V(t-s)Z_1(u)\Big\|_{L^{\infty}}ds\nnm\\
&\quad+\int_0^t\bigg\|\frac{1}{\epsilon}e^{\frac{(t-s)B_{\epsilon}}{\epsilon^2}}G_2(f_{\epsilon})-V(t-s)\div_xZ_2(u)\bigg\|_{L^{\infty}}ds\nnm\\
&\quad+\int_0^t\bigg\|\frac{1}{\epsilon^2}e^{\frac{(t-s)B_{\epsilon}}{\epsilon^2}}v\sqrt{M}\cdot\nabla_x(I-\Delta_x)^{-1}G_3(f_{\epsilon})\bigg\|_{L^{\infty}}ds\nnm\\
&=:I_1+I_2+I_3+I_4.
\ema

For $I_1$, by Lemma \ref{3dmvpbfas6}, we have
\bma\label{dlth2-2}
I_1&\leq C\bigg(\epsilon(1+t)^{-2}+\(1+\frac{t}{\epsilon}\)^{-1}\bigg)\(\|f_0\|_{H^{3}}+\|f_0\|_{W^{3,1}}\)\nnm\\
&\leq C\delta_0\bigg(\epsilon(1+t)^{-2}+\(1+\frac{t}{\epsilon}\)^{-1}\bigg).
\ema

To esitmate $I_2$, we decompose
\bma\label{dlth2-3}
I_2&\leq \int_0^t\Big\|e^{\frac{(t-s)B_{\epsilon}}{\epsilon^2}}G_1(f_{\epsilon})-V(t-s)P_0G_1(f_{\epsilon})\Big\|_{L^{\infty}}ds\nnm\\
&\quad+\int_0^t\|V(t-s)P_0G_1(f_{\epsilon})-V(t-s)Z_1(u)\|_{L^{\infty}}ds\nnm\\
&=:I_{21}+I_{22}.
\ema
By Lemma \ref{3dmvpbfas6} and \eqref{rbedl5-1-2}, we have
\bma\label{dlth2-3-1}
I_{21}&\leq C\int^t_0\(\epsilon(1+t-s)^{-2}+\(1+\frac{t-s}{\epsilon}\)^{-1}\) \|G_1(f_{\epsilon})\|_{H^{3}\cap W^{3,1}} ds\nnm\\
&\leq C\delta_0^2\int^t_0\(\epsilon(1+t-s)^{-2}+\(1+\frac{t-s}{\epsilon}\)^{-1}\)(1+s)^{-\frac{3}{2}}ds.
\ema
Note that for $t\leq1$,
\bq\label{dlth2-3-1-2-1}
\int^t_0\(1+\frac{t-s}{\epsilon}\)^{-1}(1+s)^{-\frac{3}{2}}ds\leq\int^t_0\(1+\frac{t-s}{\epsilon}\)^{-1}ds\leq C\epsilon|\ln\epsilon|,
\eq
and for $t\geq1$,
\bq\label{dlth2-3-1-2-2}
\int^t_0\(1+\frac{t-s}{\epsilon}\)^{-1}(1+s)^{-\frac{3}{2}}ds\leq \epsilon\int^t_0(t-s)^{-1}(1+s)^{-\frac{3}{2}}ds\leq C\epsilon(1+t)^{-1}.
\eq
Then, by \eqref{dlth2-3-1}-\eqref{dlth2-3-1-2-2}, we have
\be\label{dlth2-3-1-j}
I_{21} \leq C\delta_0^2\epsilon|\ln\epsilon|(1+t)^{-1}.
\ee
Since
$$
P_0G_1(f_{\epsilon})=(n_{\epsilon}\nabla_x\phi_{\epsilon})\cdot v\chi_0+\sqrt{\frac{2}{3}}(m_{\epsilon}\cdot\nabla_x\phi_{\epsilon})\chi_4,
$$
where $n_{\epsilon}=(f_{\epsilon},\chi_0)$ and $m_{\epsilon}=(f_{\epsilon},v\chi_0)$, it follows from  Lemma \ref{3dmvpbfas3} and Lemma \ref{rbedl5} that
\bma\label{dlth2-3-1-3}
I_{22}&\leq C\int^t_0(1+t-s)^{-\frac{3}{4}}\ln\(2+\frac{1}{t-s}\)\nnm\\
&\qquad \times\(\|f_{\epsilon}-u\|_{L^{\infty}}\|\nabla_x\phi_{\epsilon}\|_{H^2_{x}}+\|f_{\epsilon}\|_{H^2}\|\nabla_x\phi_{\epsilon}-\nabla_x\phi\|_{L^{\infty}_x}\)ds\nnm\\
&\leq C\delta_0\Lambda_{\epsilon}(t)\int^t_0(1+t-s)^{-\frac{3}{4}}\ln\(2+\frac{1}{t-s}\)\nnm\\
&\qquad \times\(\epsilon|\ln\epsilon|^2(1+s)^{-\frac{3}{4}}+\(1+\frac{s}{\epsilon}\)^{-1}\)(1+s)^{-\frac{3}{4}}ds.
\ema
For $t\leq\epsilon$, we have
\bma
&\quad\int^{t}_0(1+t-s)^{-\frac{3}{4}}\ln\(2+\frac{1}{t-s}\)\(1+\frac{s}{\epsilon}\)^{-1}(1+s)^{-\frac{3}{4}}ds\nnm\\
&\leq C\int^{t}_0\ln\(2+\frac{1}{t-s}\)ds\le C\leq C\(1+\frac{t}{\epsilon}\)^{-1}.\label{dlth2-3-1-5}
\ema
For $t\geq\epsilon$, we have
\bma
&\quad\(\int^{\frac{t}{2}}_0+\int_{\frac{t}{2}}^t\)(1+t-s)^{-\frac{3}{4}}\ln\(2+\frac{1}{t-s}\)\(1+\frac{s}{\epsilon}\)^{-1}(1+s)^{-\frac{3}{4}}ds\nnm\\
&\leq C(1+t)^{-\frac{3}{4}}\ln\(2+\frac{2}{t}\)\(\int^{1}_0\(1+\frac{s}{\epsilon}\)^{-1}ds+\int^{\frac{t}{2}}_1\(\frac{s}{\epsilon}\)^{-1}s^{-\frac{3}{4}}ds\)\nnm\\
&\quad+C\(1+\frac{t}{\epsilon}\)^{-1}(1+t)^{-\frac{3}{4}}\int_{\frac{t}{2}}^t(1+t-s)^{-\frac{3}{4}}\ln\(2+\frac{1}{t-s}\)ds\nnm\\
&\leq C\epsilon|\ln\epsilon|^2(1+t)^{-\frac{3}{4}}+C\(1+\frac{t}{\epsilon}\)^{-1}.\label{dlth2-3-1-7}
\ema

Thus, by \eqref{dlth2-3-1-3}-\eqref{dlth2-3-1-7}, we have
\be\label{dlth2-3-1-8}
I_{22}\leq C\delta_0\Lambda_{\epsilon}(t)\bigg(\epsilon|\ln\epsilon|^2(1+t)^{-\frac{3}{4}}+\(1+\frac{t}{\epsilon}\)^{-1}\bigg).
\ee

For $I_3$, we decompose
\bma\label{dlth2-4}
I_3&\leq\int_0^t\bigg\|\frac{1}{\epsilon}e^{\frac{(t-s)B_{\epsilon}}{\epsilon^2}}G_2(f_{\epsilon})-V(t-s)P_0\big(v\cdot\nabla_xL^{-1}G_2(f_{\epsilon})\big)\bigg\|_{L^{\infty}}ds\nnm\\
&\quad+\int_0^t\big\|V(t-s)P_0\big(v\cdot\nabla_xL^{-1}G_2(f_{\epsilon})\big)-V(t-s)\div_xZ_2\big\|_{L^{\infty}}ds\nnm\\
&=:I_{31}+I_{32}.
\ema
For $I_{31}$, by Lemma \ref{3dmvpbfas7} and \eqref{rbedl5-1-2}, and noticing that $P_0G_2(f_{\epsilon})=0$, we have
\bma
I_{31}&\leq C\int_0^t\(\epsilon(1+t-s)^{-\frac{5}{2}}+\(1+\frac{t-s}{\epsilon}\)^{-1}+\frac{1}{\epsilon}e^{-\frac{\sigma_0 (t-s)}{\epsilon^2}}\)\|G_2(f_{\epsilon})\|_{H^{4}\cap W^{4,1}} ds\nnm\\
&\leq C\delta_0^2\int_0^t\(\epsilon(1+t-s)^{-\frac{5}{2}}+\(1+\frac{t-s}{\epsilon}\)^{-1}+\frac{1}{\epsilon}e^{-\frac{\sigma_0 (t-s)}{\epsilon^2}}\)(1+s)^{-\frac{3}{2}}ds\nnm\\
&\leq C\delta_0^2\epsilon|\ln\epsilon|(1+t)^{-1},\label{dlth2-4-1}
\ema
where we had used \eqref{dlth2-3-1-2-1}-\eqref{dlth2-3-1-2-2}.

To estimate $I_{32}$, we decompose
\bmas
P_0\big(v\cdot\nabla_xL^{-1}G_2(f_{\epsilon})\big)&=P_0\big(v\cdot\nabla_xL^{-1}\Gamma(P_0f_{\epsilon},P_0f_{\epsilon})\big)+2P_0\big(v\cdot\nabla_xL^{-1}\Gamma(P_0f_{\epsilon},P_1f_{\epsilon})\big)\\
&\quad+P_0\big(v\cdot\nabla_xL^{-1}\Gamma(P_1f_{\epsilon},P_1f_{\epsilon})\big)\nnm\\
&=:J_1+J_2+J_3.
\emas
By Lemma \ref{ocr-1}, we can obtain (cf. \cite{Li-4})
$$
J_1=-\sum^3_{i,j=1}\partial_{x_i}(m_{\epsilon}^im_{\epsilon}^j)v_j\chi_0+\frac{1}{3}\sum^3_{i,j=1}\partial_{x_j}(m_{\epsilon}^i)^2v_j\chi_0-\frac{5}{3}\partial_{x_j}(m_{\epsilon}^jq_{\epsilon})\chi_4.
$$
Since
$\big(v\chi_0\cdot\xi\hat{g}_0,E_j(\xi)\big)_{\xi}=0,$ $j=0,2,3$  for any $ g_0\in L^2(\R^3_x),$
we have by \eqref{lns4j1} that
\be\label{lsp2j3-3}
V(t)(v\chi_0\cdot\Tdx g_0)=\sum_{j=0,2,3}\mathcal{F}^{-1}\[e^{-d_jt}\big(v\chi_0\cdot\xi\hat{g}_0,E_j(\xi)\big)_{\xi}E_j(\xi)\]=0.
\ee
Thus,
\bma \label{dlth2-4-6}
V(t-s)J_1(s)&=-V(t-s)\div_x\[(m_{\epsilon}\otimes m_{\epsilon})\cdot v\chi_0+\frac{5}{3}(m_{\epsilon}q_{\epsilon})\chi_4\]\nnm\\
&=:V(t-s)\div_xJ_4(s).
\ema
By Lemma \ref{3dmvpbfas3}, Lemmas \ref{rbedl5}-\ref{rbedl6} and \eqref{dlth2-4-6}, we can obtain
\bma\label{dlth2-4-7}
I_{32}&\leq \int_0^t\|V(t-s)\div_x(J_4-Z_2)(s)\|_{L^{\infty}}ds+C\sum^3_{k=2}\int_0^t\|V(t-s)J_k(s)\|_{L^{\infty}}ds\nnm\\
&\leq C\int_0^t(1+t-s)^{-\frac{3}{4}}(t-s)^{-\frac{1}{2}}\|f_{\epsilon}-u\|_{L^{\infty}}(\|f_{\epsilon}\|_{H^2}+\|u\|_{H^2})ds\nnm\\
&\quad+C\int_0^t(1+t-s)^{-\frac{3}{4}}(t-s)^{-\frac{1}{2}}\|P_1f_{\epsilon}\|_{H^2}(\|P_0f_{\epsilon}\|_{H^2}+\|P_1f_{\epsilon}\|_{H^2})ds\nnm\\
&\leq C\delta_0\Lambda_{\epsilon}(t)\int_0^t(1+t-s)^{-\frac{3}{4}}(t-s)^{-\frac{1}{2}}\nnm\\
&\qquad \times\(\epsilon|\ln\epsilon|^2(1+s)^{-\frac{3}{4}}+\(1+\frac{s}{\epsilon}\)^{-1}\)(1+s)^{-\frac{3}{4}}ds\nnm\\
&\quad+C\delta_0^2\int_0^t(1+t-s)^{-\frac{3}{4}}(t-s)^{-\frac{1}{2}}\(\epsilon(1+s)^{-\frac{5}{4}}+e^{-\frac{\sigma_0 s}{2\epsilon^2}}\)(1+s)^{-\frac{3}{4}}ds.
\ema
For $t\leq\epsilon$,
\bma
&\quad\int^t_{0}(1+t-s)^{-\frac{3}{4}}(t-s)^{-\frac{1}{2}}\(1+\frac{s}{\epsilon}\)^{-1}(1+s)^{-\frac{3}{4}}ds\nnm\\
&\leq C\int_0^t(t-s)^{-\frac{1}{2}}ds\leq   C\le C\(1+\frac{t}{\epsilon}\)^{-1}.\label{dlth2-4-7-2}
\ema
For $t\geq\epsilon$,
\bma
&\quad\(\int_0^{\frac{t}{2}}+\int^t_{\frac{t}{2}}\)(1+t-s)^{-\frac{3}{4}}(t-s)^{-\frac{1}{2}}\(1+\frac{s}{\epsilon}\)^{-1}(1+s)^{-\frac{3}{4}}ds\nnm\\
&\leq C\epsilon |\ln\epsilon|(1+t)^{-\frac{3}{4}}t^{-\frac12}+C\(1+\frac{t}{\epsilon}\)^{-1}(1+t)^{-\frac{3}{4}}\sqrt{t}\nnm\\
&\leq C\bigg(\epsilon|\ln\epsilon|^2(1+t)^{-1}+\(1+\frac{t}{\epsilon}\)^{-1}\bigg).\label{dlth2-4-7-6}
\ema
Thus, by \eqref{dlth2-4-7}-\eqref{dlth2-4-7-6}, we have
\be\label{dlth2-4-8}
I_{32} \leq C\(\delta_0^2+\delta_0\Lambda_{\epsilon}(t)\)\bigg(\epsilon|\ln\epsilon|^2(1+t)^{-1}+\(1+\frac{t}{\epsilon}\)^{-1}\bigg).
\ee

For $I_4$, by Lemma \ref{3dmvpbfas6}, \eqref{lsp2j3-3} and \eqref{youngbds}, we have
\bma\label{dlth2-4-9}
I_4&\leq \int_0^t\bigg\|\frac{1}{\epsilon^2}e^{\frac{(t-s)B_{\epsilon}}{\epsilon^2}}v\sqrt{M}\cdot\nabla_x(I-\Delta_x)^{-1}G_3(s) \bigg\|_{L^{\infty}}ds\nnm\\
&\leq \frac{C}{\epsilon^2}\int_0^t\bigg(\epsilon(1+t-s)^{-\frac{5}{2}}+\(1+\frac{t-s}{\epsilon}\)^{-1}\bigg)  \|\nabla_x(I-\Delta_x)^{-1}G_3(s) \|_{H^3_x\cap W^{3,1}_x} ds\nnm\\
&\leq C\int_0^t\bigg(\epsilon(1+t-s)^{-\frac{5}{2}}+\(1+\frac{t-s}{\epsilon}\)^{-1}\bigg) \|\phi_{\epsilon}(s)\|^2_{H^3_x} ds\nnm\\
&\leq C\delta_0^2\epsilon|\ln\epsilon|(1+t)^{-1},
\ema
where we have used \eqref{dlth2-3-1-2-1}-\eqref{dlth2-3-1-2-2}.

By combining \eqref{dlth2-1}-\eqref{dlth2-3}, \eqref{dlth2-3-1-j}, \eqref{dlth2-3-1-8}-\eqref{dlth2-4-1} and \eqref{dlth2-4-8}-\eqref{dlth2-4-9}, we have
\be\label{dlth2-5}
\|f_{\epsilon}(t)-u(t)\|_{L^{\infty}}\leq C\(\delta_0+\delta_0^2+\delta_0\Lambda_{\epsilon}(t)\)\bigg(\epsilon|\ln\epsilon|^2(1+t)^{-\frac{3}{4}}+\(1+\frac{t}{\epsilon}\)^{-1}\bigg).
\ee

By Lemma \ref{rbedl5}, \eqref{youngbds} and \eqref{dlth2-5}, we have
\bma\label{dlth2-4-10}
\|\nabla_x\phi_{\epsilon}-\nabla_x\phi\|_{L^{\infty}_x}
&\leq C\|f_{\epsilon}(t)-u(t)\|_{L^{\infty}}+\frac{1}{\epsilon} \|\nabla_x(I-\Delta_x)^{-1}G_3(f_{\epsilon}) \|_{L^{\infty}_x}\nnm\\
&\leq C\|f_{\epsilon}(t)-u(t)\|_{L^{\infty}}+C\epsilon\|\phi_{\epsilon}\|_{L^{\infty}_x}^2\nnm\\
&\leq C\(\delta_0+\delta_0^2+\delta_0\Lambda_{\epsilon}(t)\)\bigg(\epsilon|\ln\epsilon|^2(1+t)^{-\frac{3}{4}}+\(1+\frac{t}{\epsilon}\)^{-1}\bigg).
\ema

By combining \eqref{dlth2-5} and \eqref{dlth2-4-10}, we obtain
\bq
\Lambda_{\epsilon}(t)\leq C\delta_0+C\delta_0^2+C\delta_0\Lambda_{\epsilon}(t),
\eq
where $C>0$ is a constant independent of $\epsilon$. By taking $\delta_0>0$ small enough, we can obtain \eqref{dlth2-0}, which
proves \eqref{thm-2-1}.

Finally, we prove \eqref{thm-2-2}. Define
\bq\label{dlth-2-2}
\Omega_{\epsilon}(t)=\sup_{0\leq s\leq t}\(\epsilon|\ln\epsilon|\)^{-1}(1+s)^{\frac{3}{4}}\(\|f_{\epsilon}(s)-u(s)\|_{L^{\infty}}+\|\nabla_x\phi_{\epsilon}(s)-\nabla_x\phi(s)\|_{L^{\infty}_x}\),
\eq
If $f_0$ satisfies \eqref{F01}, then we have by Lemma \ref{3dmvpbfas6} that
\be
I_1\leq C\delta_0\epsilon(1+t)^{-2}.\label{dlth-2-2-1}
\ee
For $I_{21}$, $I_{31}$ and $I_{4}$, we have by \eqref{dlth2-3-1}, \eqref{dlth2-4-1} and \eqref{dlth2-4-9} that
\be\label{dlth-jre-i2314}
I_{21}+I_{31}+I_4\le C\delta_0^2\eps|\ln \eps|(1+t)^{-1}.
\ee
For $I_{22}$ and $I_{32}$, we have by Lemma \ref{3dmvpbfas3} and \eqref{rbedl5-1-2} that
\bma
I_{22}&\leq C\delta_0\Omega_{\epsilon}(t)\epsilon|\ln\epsilon|\int^t_0(1+t-s)^{-\frac{3}{4}}\ln\(2+\frac{1}{t-s}\)(1+s)^{-\frac{3}{2}}ds\nnm\\
&\leq C\delta_0\Omega_{\epsilon}(t)\epsilon|\ln\epsilon|(1+t)^{-\frac{3}{4}},\label{dlth-2-2-3}\\
I_{32}&\leq C\delta_0\Omega_{\epsilon}(t)\epsilon|\ln\epsilon|\int_0^t(1+t-s)^{-\frac{3}{4}}(t-s)^{-\frac{1}{2}}(1+s)^{-\frac{3}{2}}ds\nnm\\
&\quad+C\delta_0^2\int_0^t(1+t-s)^{-\frac{3}{4}}(t-s)^{-\frac{1}{2}}\(\epsilon(1+s)^{-\frac{5}{4}}+e^{-\frac{\sigma_0 s}{2\epsilon^2}}\)(1+s)^{-\frac{3}{4}}ds\nnm\\
&\leq C\(\delta_0\Omega_{\epsilon}(t)+\delta_0^2\)\epsilon|\ln\epsilon|(1+t)^{-\frac{3}{4}},\label{dlth-2-2-5}
\ema
where we have used \eqref{dlth2-3-1-2-1}-\eqref{dlth2-3-1-2-2}.

By combining \eqref{dlth-2-2-1}-\eqref{dlth-2-2-5}, we have
\be\label{dlth2-5a}
\|f_{\epsilon}(t)-u(t)\|_{L^{\infty}}\leq C\(\delta_0+\delta_0^2+\delta_0\Omega_{\epsilon}(t)\) \epsilon|\ln\epsilon| (1+t)^{-\frac{3}{4}} .
\ee

By  \eqref{dlth2-5a} and \eqref{youngbds}, we have
\bmas
\|\nabla_x\phi_{\epsilon}-\nabla_x\phi\|_{L^{\infty}_x}&\leq C\|f_{\epsilon}(t)-u(t)\|_{L^{\infty}}+\frac{1}{\epsilon} \|\nabla_x(I-\Delta_x)^{-1}G_3(f_{\epsilon}) \|_{L^{\infty}_x}\\
&\leq C\(\delta_0+\delta_0^2+\delta_0\Omega_{\epsilon}(t)\)\epsilon|\ln\epsilon|(1+t)^{-\frac{3}{4}}+C\delta_0^2\epsilon(1+t)^{-\frac{3}{2}}.
\emas
Thus, we can obtain
$$
\Omega_{\epsilon}(t)\leq C\delta_0+C\delta_0^2+C\delta_0\Omega_{\epsilon}(t),
$$
where $C>0$ is a constant independent of $\epsilon$. By taking $\delta_0>0$ small enough, we can obtain \eqref{dlth-2-2}, which
proves \eqref{thm-2-2}. The proof is completed.

\medskip
\noindent {\bf Acknowledgements:} The research of this work was supported  by the special foundation for Guangxi Ba Gui Scholars,  and the National Natural Science Foundation of China  grants (No.  12171104).


\end{document}